\documentclass[11pt]{amsart}

\usepackage{amscd}
\usepackage{amsmath, amssymb}
\usepackage{amsfonts}
\newcommand{\de}{\partial}
\newcommand{\db}{\overline{\partial}}

\newcommand{\ddt}{\frac{\partial}{\partial t}}
\newcommand{\ddb}{\partial \ov{\partial}}

\newcommand{\ddbar}{\sqrt{-1} \partial \overline{\partial}}
\newcommand{\Ric}{\mathrm{Ric}}
\newcommand{\ov}[1]{\overline{#1}}
\newcommand{\mn}{\sqrt{-1}}

\newcommand{\tr}[2]{\textrm{tr}_{#1}{#2}}
\newcommand{\ti}[1]{\tilde{#1}}
\newcommand{\vp}{\varphi}

\newcommand{\ve}{\varepsilon}

\renewcommand{\leq}{\leqslant}
\renewcommand{\geq}{\geqslant}
\renewcommand{\le}{\leqslant}
\renewcommand{\ge}{\geqslant}

\begin{document}
\newcounter{remark}
\newcounter{theor}
\setcounter{remark}{0}
\setcounter{theor}{1}
\newtheorem{claim}{Claim}
\newtheorem{theorem}{Theorem}[section]
\newtheorem{lemma}[theorem]{Lemma}
\newtheorem{corollary}[theorem]{Corollary}
\newtheorem{proposition}[theorem]{Proposition}
\newtheorem{question}{question}[section]
\newtheorem{defn}{Definition}[theor]
\numberwithin{equation}{section}
\theoremstyle{definition}
\newtheorem{remark}[theorem]{Remark}

\newenvironment{example}[1][Example]{\addtocounter{remark}{1} \begin{trivlist}
\item[\hskip
\labelsep {\bfseries #1  \thesection.\theremark}]}{\end{trivlist}}
\title[Monge-Amp\`ere equation for $\MakeLowercase{(n-1)}$-PSH functions]{The Monge-Amp\`ere equation for $(n-1)$-plurisubharmonic functions on a compact K\"ahler manifold$^*$}

\author[V. Tosatti]{Valentino Tosatti}
\thanks{$^{*}$Research supported in part by NSF grants    DMS-1236969 and DMS-1105373.  The first named-author is supported in part by a Sloan Research Fellowship.}
\address{Department of Mathematics, Northwestern University, 2033 Sheridan Road, Evanston, IL 60208}
\author[B. Weinkove]{Ben Weinkove}
\dedicatory{Dedicated to Professor Duong H. Phong on the occasion of his 60th birthday}
\begin{abstract}   A $C^2$ function on $\mathbb{C}^n$ is called $(n-1)$-plurisubharmonic in the sense of Harvey-Lawson if the sum of any $n-1$ eigenvalues of its complex Hessian is nonnegative.  We show that the associated Monge-Amp\`ere equation  can be solved on any compact K\"ahler manifold.  As a consequence we prove the existence of solutions to an equation of Fu-Wang-Wu, giving Calabi-Yau theorems for balanced, Gauduchon and strongly Gauduchon metrics on compact K\"ahler manifolds.
 \end{abstract}

\maketitle

\section{Introduction}

A $C^2$ function $u$ on a domain $\Omega \subset \mathbb{C}^n$ is called \emph{plurisubharmonic} if the $n\times n$ Hermitian matrix  with $(i,j)$th entry equal to
$$\frac{\partial^2 u}{\partial z^i \partial \ov{z}^j}$$
is nonnegative definite. Equivalently, $\ddbar u$ is a nonnegative $(1,1)$ form.
Another equivalent definition is that $u$ is subharmonic when restricted to every complex line intersecting $\Omega$.
Harvey-Lawson \cite{HL2, HL} introduced a more general notion of \emph{$k$-plurisubharmonicity}, where the complex lines in this second definition are replaced by complex $k$-planes.

In this paper we are interested in the case of $(n-1)$-plurisubharmonic functions (for $n\geq 2$), $(n-1)$-PSH for short, which can be characterized as follows.  A $C^2$ function $u$ is  $(n-1)$-PSH if the matrix whose $(i,j)$th entry is
\begin{equation}\label{matrice}
\left( \sum_{k=1}^n \frac{\partial^2 u}{\partial z^k \partial \ov{z}^k} \right) \delta_{ij} -  \frac{\partial^2 u}{\partial z^i \partial \ov{z}^j},
\end{equation}
is nonnegative definite. Equivalently, $(\Delta u) \beta - \ddbar u$ is a nonnegative $(1,1)$ form where $\beta=\mn\sum_i dz^i\wedge d\ov{z}^i$ is the Euclidean K\"ahler form and $\Delta$ the associated Laplace operator on functions.  Other equivalent definitions are that
 $\ddbar u\wedge\beta^{n-2}$ is a nonnegative $(n-1,n-1)$ form, or
 that the sum of any $(n-1)$ eigenvalues of the complex Hessian $\frac{\partial^2 u}{\partial z^i \partial \ov{z}^j}$ is nonnegative.  $(n-1)$-PSH functions were also recently considered by Han-Ma-Wu \cite{HMW2}.

If the complex Hessian is replaced by the real Hessian, such functions are sometimes called ``$(n-1)$-convex'' (or with ``$(n-1)$-positive Hessian''), and were first studied by Sha \cite{Sh} and Wu \cite{Wu}.  For smooth and positive right hand side, the Dirichlet problem  for $(n-1)$-convex functions on suitable domains, was solved in the work of
Caffarelli-Nirenberg-Spruck \cite{CNS}, where they treated a general class of nonlinear equations.

Taking the determinant of the matrix in \eqref{matrice} defines a Monge-Amp\`ere equation for $(n-1)$-PSH functions:
\begin{equation} \label{manew}
 \det \left( \left( \sum_{k=1}^n \frac{\partial^2 u}{\partial z^k \partial \ov{z}^k} \right) \delta_{ij} -  \frac{\partial^2 u}{\partial z^i \partial \ov{z}^j} \right) =f.
\end{equation}
 In complex dimension 2, the equation (\ref{manew}) is equivalent to the usual complex Monge-Amp\`ere equation, so the results of  this paper are only new for dimension $n\ge 3$.  For some recent developments in the theory of the complex Monge-Amp\`ere equation, see \cite{PSS} and the references therein.

The  Dirichlet problem for (\ref{manew}) on suitable domains in $\mathbb{C}^n$ for positive $f$ was solved by Song-Ying Li \cite{Li} who, as in \cite{CNS}, considered a general class of operators.
Harvey-Lawson investigated  (\ref{manew})   explicitly (see e.g. \cite[Example 4.3.2]{HL3}) and solved the Dirichlet problem with $f=0$ on suitable domains.

In this paper we solve the corresponding Monge-Amp\`ere equation on compact K\"ahler manifolds.
Let $(M,g)$ be a compact K\"ahler manifold of complex dimension $n \ge 2$ and write $\omega = \sqrt{-1} g_{i\ov{j}} dz^i \wedge d\ov{z}^j$ for the associated K\"ahler form.   Write $\Delta = g^{i\ov{j}} \partial_i \partial_{\ov{j}}$ for the associated Laplace operator on functions.  We let $h$ be a  Hermitian metric on $M$ (for our applications it will in general be non-K\"ahler), and write $\omega_h$ for its associated real $(1,1)$ form.

Our main result is as follows:

\begin{theorem} \label{maintheorem} Let $(M,\omega)$ be an $n$-dimensional compact K\"ahler manifold, $n\geq 2$, and let $\omega_h$ be a Hermitian metric
and $F$ be a smooth function on $M$.  Then there exists a unique pair $(u,b)$ where $u$ is a smooth real function and $b$ is a constant, solving
\begin{equation} \label{MA1}
\left( \omega_h + \frac{1}{n-1} ( (\Delta u) \omega - \ddbar u) \right)^n = e^{F+b} \omega^n,
\end{equation}
with
\begin{equation} \label{condthm}
\omega_h + \frac{1}{n-1} (( \Delta u)\omega - \ddbar u) > 0, \quad \sup_M u=0.
\end{equation}
\end{theorem}

We can restate this theorem in terms of $(n-1, n-1)$ forms.  The equation (\ref{MA1}) becomes the ``form-type'' Calabi-Yau equation of Fu-Wang-Wu \cite{FWW2, FWW}.

\begin{corollary} \label{cor1} Let $(M, \omega)$ be a compact K\"ahler manifold and
 $g_0$  a Hermitian metric on $M$ with associated $(1,1)$ form $\omega_0$.  Let $F$ be a smooth function.  Then there exists a unique pair $(u,b)$ where $u$ is a smooth function on $M$ and $b$ a constant, solving
\begin{equation} \label{FWW}
\det \left( \omega_0^{n-1} + \ddbar u \wedge \omega^{n-2} \right) = e^{F+b} \det (\omega^{n-1}),
\end{equation}
with
\begin{equation} \label{condcor1}
\omega_0^{n-1} + \ddbar u \wedge \omega^{n-2} >0, \quad \sup_M u =0.
\end{equation}
\end{corollary}

In fact, if $\omega_0$ and $\omega_h$ are related by $(n-1)! \, \omega_h = * \omega_0^{n-1}$, where $*$ is the Hodge star operator of $g$, then  (\ref{MA1}) and (\ref{FWW}) are completely equivalent. This is explained below in Section \ref{prelims}.

It was shown by Fu-Wang-Wu \cite{FWW} that (\ref{FWW}) (and hence also (\ref{MA1})) can be solved under the assumption that $\omega$ has nonnegative orthogonal bisectional curvature.  Hence the main content of this paper is to remove this assumption on curvature.  Corollary \ref{cor1} is implicitly conjectured in \cite{FWW2, FWW} and the expectation is that more generally the equation
\begin{equation} \label{eco}
\det \left( \omega_0^{n-1} + \ddbar ( u  \omega^{n-2}) \right) = e^{F+b} \det (\omega^{n-1}),
\end{equation}
has solutions for non-K\"ahler $\omega$.

Let us also mention another application of our results to the ``Strominger system''. This is a system of nonlinear PDEs on a compact complex manifold introduced by Strominger \cite{St}, which can be seen as a generalization of the Calabi-Yau equation to non-K\"ahler metrics. Nontrivial solutions of the Strominger system were first constructed by Li-Yau \cite{LY} on some K\"ahler manifolds, and by Fu-Yau \cite{FY} on certain non-K\"ahler manifolds.
As explained in \cite{FWW}, Corollary \ref{cor1} allows us to find non-K\"ahler solutions of a strengthening of one of the equations in the Strominger system on Calabi-Yau K\"ahler manifolds \cite[(2.5), (13.4)]{FY}, \cite{LY}, namely
\begin{equation} \label{strom}
d\left(|\Omega|_\omega \omega^{n-1}\right)=0,
\end{equation}
where $\Omega$ is a never-vanishing holomorphic $n$-form on $M$. Indeed, we can use Corollary \ref{cor1} (or Corollary \ref{corbal}) to produce a Hermitian metric $\omega$ with $d(\omega^{n-1})=0$ and $|\Omega|_\omega$ a constant, which implies in particular that \eqref{strom} is satisfied.

As another consequence of Corollary \ref{cor1} we obtain ``Calabi-Yau'' theorems on a compact K\"ahler manifold for balanced, Gauduchon and strongly Gauduchon metrics. Recall that a metric $\omega_0$ is \emph{balanced} \cite{Mi} (or semi-K\"ahler) if $d(\omega_0^{n-1})=0$ and \emph{Gauduchon} \cite{Ga} if $\partial \ov{\partial} (\omega_0^{n-1})=0$. More recently, Popovici \cite{Po2} introduced the notion of {\em strongly Gauduchon} metrics, for which $\db(\omega_0^{n-1})$ is $\de$-exact.

If $\omega_0$ is balanced then the $(n-1)$th root of $\omega_0^{n-1} + \ddbar u \wedge \omega^{n-2}$  is balanced, and the same holds for Gauduchon and strongly Gauduchon.  Hence (for more details, see Section \ref{prelims}):

\begin{corollary} \label{corbal}  Let $(M, \omega)$ be a compact K\"ahler manifold.
Let $\omega_0$ be a balanced (resp., Gauduchon, resp., strongly Gauduchon) metric on $M$ and $F'$ a smooth function on $M$.  Then there exists a unique constant $b'$ and a unique  balanced (resp., Gauduchon, resp., strongly Gauduchon) metric, which we write as $\omega_u$, with
\begin{equation} \label{mat}
\omega_u^{n-1} = \omega_0^{n-1} + \ddbar u \wedge \omega^{n-2},
\end{equation}
for some smooth function $u$, solving the Calabi-Yau equation
\begin{equation}\label{ma}
\omega_u^n = e^{F'+b'} \omega^n.
\end{equation}
\end{corollary}

In particular this answers affirmatively an explicit conjecture of Fu-Xiao \cite[p.3]{FX} and a question of Popovici \cite[Question 1.4]{Po} in the case of K\"ahler manifolds.  Popovici conjectures an analogous statement for balanced non-K\"ahler manifolds.  As in the approach of Fu-Wang-Wu \cite{FWW2, FWW}
one can
replace (\ref{mat}) by
$$\omega_u^{n-1} = \omega_0^{n-1} + \ddbar (u \omega^{n-2}),$$
and ask if Corollary \ref{corbal} holds for $\omega$ Hermitian.

  The authors are grateful to J.-P. Demailly for drawing their attention to the equations  (\ref{FWW}), (\ref{ma}) and their connection to Popovici's  strongly Gauduchon metrics.

Finally, as in the case of the Calabi-Yau theorem \cite{Ya}, we can interpret \eqref{ma} as a prescribed Ricci curvature equation for balanced metrics.
Recall that for a Hermitian metric $\omega$, we can define its first Chern form locally by $\Ric^{\mathrm{C}}(\omega)=-\ddbar \log\det (g)$,
which is a global closed real $(1,1)$ form cohomologous to $c_1(M)$, and its components equal one of the Ricci curvatures of the Chern connection of $\omega$.

Following \cite{Bi}, we consider the Bismut connection of $\omega$, and define its Ricci form $\Ric^{\mathrm{B}}(\omega)$. Then (see e.g. \cite[(2)]{FG}) we have the relation
$$\Ric^{\mathrm{B}}(\omega)=\Ric^{\mathrm{C}}(\omega)+dd^*\omega.$$
In particular, if $\omega$ is balanced, then $$d^*\omega=- *d*\omega= - \frac{1}{(n-1)!} *d(\omega^{n-1})=0,$$
and so the Chern and Bismut Ricci curvatures agree in this case. Now taking $\ddb\log$ of \eqref{ma} gives
$$\Ric^{\mathrm{C}}(\omega_u)=\Ric^{\mathrm{C}}(\omega)-\ddbar F',$$ and we conclude the following:
\begin{corollary}
Let $(M,\omega)$ be a compact K\"ahler manifold and $\omega_0$ be a balanced metric on $M$.
If $\psi$ is a closed real $(1,1)$ form cohomologous to $c_1(M)$, then there exists a unique balanced metric $\omega_u$ on $M$ with
$$\omega_u^{n-1} = \omega_0^{n-1} + \ddbar u \wedge \omega^{n-2},$$
for some smooth function $u$, and solving
$$\Ric^{\mathrm{B}}(\omega_u)=\Ric^{\mathrm{C}}(\omega_u)=\psi.$$
\end{corollary}
It is easy to see that every compact K\"ahler manifold of dimension $n\geq 3$ has a non-K\"ahler balanced metric.
Therefore, if $c_1(M)=0$ in $H^2(M,\mathbb{R})$, i.e. $M$ is Calabi-Yau and K\"ahler (of dimension $n\geq 3$), then it follows that there exist many non-K\"ahler balanced metrics on $M$ with vanishing Chern and Bismut Ricci curvatures. This means that the restricted holonomy of the Chern and Bismut connections are both contained in $\textrm{SU}(n)$, and such ``balanced Calabi-Yau metrics with torsion'' have received much attention in the mathematical physics literature, see e.g. \cite{FG, Fr, Fu, FLY, FX, FY, Gr, GIP}.

In a separate development, there has also been recent interest in balanced metrics
\cite{To, FX} (see also \cite{BDPP}) in relation to the  cone they define in $H^{n-1,n-1}(M, \mathbb{R})$.

We now briefly  outline the proof of Theorem \ref{maintheorem}.   Rather than following the method of proof of the complex Monge-Amp\`ere equation
$$(\omega+\ddbar u)^n = e^{F+b} \omega^n,$$
which was solved by Yau \cite{Ya} when $\omega$ is K\"ahler and by  the authors \cite{TW2} for $\omega$ Hermitian (see also Cherrier \cite{Ch} and Guan-Li \cite{GL}), we instead take an approach used for  the complex Hessian equations. The  existence of a solution to the complex Hessian equations
\begin{equation} \label{CH}
(\omega+ \ddbar u)^k \wedge \omega^{n-k} = e^F \omega^n, \textrm{ for fixed } 2 \le k \le n-1,
\end{equation}
was proved by Dinew-Ko{\l}odziej \cite{DK}, who used the second order estimate of Hou-Ma-Wu \cite{HMW} (see also \cite{Bl, DK2, Hou, J, K, Li, Lu, N, Zh} for earlier and related work on these equations).  In our discussion, we point out some similarities and differences between their arguments and ours. One difference is that equation (\ref{MA1}) has two reference metrics - one K\"ahler and one Hermitian - as opposed to the single reference K\"ahler metric $\omega$ in (\ref{CH}).  Another is that the $(k,k)$ forms $(\omega+ \ddbar u)^k$ are closed if $\omega$ is closed, whereas the (1,1) form $(\Delta u) \omega - \ddbar u$ is only closed if $u$ is constant.

 As a first step, in Section \ref{linfty}, we establish an $L^{\infty}$ estimate for the function $u$.  In contrast to the case of the complex Hessian equations,  the $L^{\infty}$ estimate appears to require substantial work beyond the Yau argument \cite{Ya}.  A key ingredient is the
 pointwise estimate
 $$\ddbar u \wedge (2\omega_0^{n-1} + \ddbar u \wedge \omega^{n-2} ) \le C \omega^n,$$
 for $\omega_0$ defined by $\omega_0^{n-1} = (n-1)! * \omega_h$.  This inequality makes crucial use of the Monge-Amp\`ere equation (\ref{MA1}).
We use this to prove a ``Cherrier-type'' inequality
$$\int_M | \partial e^{-\frac{pu}{2}}|^2_g \omega^n \le C p \int_M e^{-pu} \omega^n, \quad \textrm{for } p\ge 1,$$
and then apply arguments of \cite{TW2} to obtain the desired $L^{\infty}$ estimate.

Next, in Section \ref{soe}, following the Hou-Ma-Wu \cite{HMW} estimate for the complex Hessian equations, we show that
\begin{equation} \label{hmw1}
\Delta u \le C (\sup_M | \nabla u|^2_g +1).
\end{equation}
This estimate uses a maximum principle argument similar to that in \cite{HMW}, and indeed the key quantity we differentiate is almost identical.  However the reader will find a number of differences in the two computations.  In particular, we do not need the inequalities for elementary symmetric polynomials, such as those of Chou-Wang \cite{CW} (see also \cite{TrW}), used in \cite{HMW}.

It was pointed out by Hou-Ma-Wu \cite{HMW} that the inequality (\ref{hmw1}) is compatible with a blow-up type argument, and indeed this was exactly the approach taken later by Dinew-Ko{\l}odziej \cite{DK}.  A  Liouville Theorem was proved in \cite{DK} for solutions of the homogeneous complex Hessian equations on $\mathbb{C}^n$ which, with (\ref{hmw1}), gives a uniform bound on $|\nabla u|^2$.

We apply the same strategy, and in, Section \ref{foe}, obtain the following Liouville Theorem:  if $u :\mathbb{C}^n \rightarrow \mathbb{R}$ is $(n-1)$-PSH and solves the corresponding homogeneous Monge-Amp\`ere equation, then with suitable regularity assumptions, it must be constant.  Some of the properties of the operator
$$u \mapsto (\ddbar u)^k \wedge \omega^{n-k}$$
do not seem to carry over immediately here, and so we need to develop the necessary theory to push through the Dinew-Ko{\l}odziej argument in our case. For example, as opposed to the case of plurisubharmonic functions \cite{BT}, it is not clear in general how to define the Monge-Amp\`ere operator
$$((\Delta u) \beta- \ddbar u)^n,$$
 for $(n-1)$-PSH functions in $\mathbb{C}^n, n\geq 3$, in the sense of pluripotential theory.

Combining the above gives uniform zero, first and second order estimates for $u$.  Given these bounds, the higher order estimates and continuity method of Fu-Wang-Wu \cite{FWW2, FWW} for the equation (\ref{FWW}) are already enough to establish Theorem \ref{maintheorem}.  However, for the sake of completeness, we provide our own proofs of these last  steps in Section \ref{sectionproof}, which follow in the spirit of the rest of this paper.

    We end the introduction with a remark:

 \begin{remark}  It may be interesting to consider the parabolic version of equation (\ref{MA1}):
\begin{equation} \label{pma}
\ddt{} u = \log \frac{\left( \omega_h + \frac{1}{n-1} ( (\Delta u) \omega - \ddbar u) \right)^n}{\Omega},
\end{equation}
for a fixed volume form $\Omega$.  It is reasonable to conjecture that, analogous to results of H.-D. Cao \cite{Cao} and Gill \cite{Gi} on  parabolic complex Monge-Amp\`ere equations, solutions to (\ref{pma}) exist for all time and converge (after a suitable normalization) to give solutions to (\ref{MA1}). More generally, allowing $\omega_h$ to vary in time, one can consider the evolution equation of $(1,1)$ forms $\omega_t$ given by\begin{equation} \label{flow}
\ddt{} \omega_t^{n-1} = - (n-1) \textrm{Ric}^{\textrm{C}}(\omega_t) \wedge \omega^{n-2}.
\end{equation}
On a K\"ahler manifold $(M, \omega)$,  this flow preserves the balanced, Gauduchon and strongly Gauduchon conditions.  In the case $n=2$, the flow (\ref{flow}) coincides with the Chern-Ricci flow  (see \cite{Gi, TW3, TW4, SW, TW5}) and the K\"ahler-Ricci flow if the initial metric $\omega_0$ is K\"ahler.  By considering the parabolic analogue of (\ref{eco}), one could also hope to extend these ideas to the case when $\omega$ is non-K\"ahler.
\end{remark}

 \bigskip
 \noindent
 {\bf Acknowledgements.} \  The authors thank L. Caffarelli, M. P\u{a}un, J. Song and N.S. Trudinger for some helpful conversations. The second-named author thanks J.-P. Demailly for an informative discussion which ultimately led the authors to investigate these equations.  In addition, the authors are grateful to S.-T. Yau for numerous discussions about balanced metrics over the last several years.

The authors wish to express their deep gratitude to  Professor Duong H. Phong for his unwavering support, encouragement and guidance.  His perspective and  insight continues to be a driving influence in their work.

\section{Preliminaries and Notation} \label{prelims}

We first discuss some notation.   Let $(M,g)$ be K\"ahler, as in the introduction.  Given a real $(n-1, n-1)$ form $\psi$ on $M$, we say that $\psi$ is positive definite (see for example \cite{De}) if
$$\psi \wedge \mn\gamma \wedge \ov{\gamma} \ge 0, \ \textrm{for all  (1,0)-forms $\gamma$}, $$
with equality if and only if $\gamma=0$.

We define the determinant of a real $(n-1, n-1)$ form $\psi$ to be the determinant of the matrix $\Psi_{i\ov{j}}$ given by
\begin{equation*}
\begin{split}
\psi = {} &  (\sqrt{-1})^{n-1} (n-1)! \sum_{i,j} (\textrm{sgn}(i,j)) \Psi_{i\ov{j}} dz^1 \wedge d\ov{z}^1 \wedge \cdots \wedge \widehat{dz^i} \wedge d\ov{z}^i \wedge \cdots \\
& \wedge dz^j \wedge \widehat{d\ov{z}^j} \wedge \cdots \wedge dz^n \wedge d\ov{z}^n.
\end{split}
\end{equation*}
where $\textrm{sgn}(i,j) = -1$ if $i >j$ and is equal to $1$ otherwise (cf. \cite{FWW2, FWW}).  In particular,
\begin{equation} \label{detf}
\det (\omega^{n-1}) = (\det g)^{n-1}.
\end{equation}
Given a positive definite $(n-1, n-1)$ form $\psi$, we can define its $(n-1)$th root to be the unique Hermitian metric $\omega_0$ satisfying $\omega_0^{n-1} = \psi$, see e.g. \cite{Mi}.

On the other hand, the Hodge star operator $*$ with respect to $g$ takes any $(n-1, n-1)$ form (not necessarily positive definite) to a $(1,1)$ form, and vice versa.
With the conventions of, for example \cite{MK}, $*$ is defined as follows.  If $\psi$ is a $(p,q)$-form then we define $*\psi$ to be the $(n-q, n-p)$ form satisfying the properties
\begin{enumerate}
\item[(i)] $\displaystyle{ \varphi \wedge * \ov{\psi} = \langle \varphi, \psi \rangle_g \frac{\omega^n}{n!}}$ for all $(p,q)$ forms $\varphi$.
\item[(ii)] $\displaystyle{\ov{*\psi} = * \ov{\psi}}$.
\item[(iii)] $\displaystyle{**\psi = (-1)^{p+q} \psi}$.
\end{enumerate}
Here, the pointwise inner product $\langle \ , \ \rangle_g$ is defined by
$$ \langle \varphi, \psi \rangle_g = \frac{1}{p!q!} \sum g^{\alpha_1 \ov{\lambda}_1} \cdots g^{\alpha_p \ov{\lambda}_p} g^{\mu_1 \ov{\beta}_1} \cdots g^{\mu_q \ov{\beta}_q} \varphi_{\alpha_1\cdots \alpha_p \ov{\beta}_1 \cdots \ov{\beta}_q} \ov{\psi_{\lambda_1 \cdots \lambda_p \ov{\mu}_1 \cdots \ov{\mu}_q}},$$
where the sum is taken over all repeated indices.
We are writing
 $$\varphi = \frac{1}{p!q!} \sum \varphi_{\alpha_1 \cdots \alpha_p \ov{\beta}_1 \cdots \ov{\beta}_q} dz^{\alpha_1} \wedge \cdots \wedge dz^{\alpha_p} \wedge d\ov{z}^{\beta_1} \wedge \cdots \wedge d\ov{z}^{\beta_q}$$
 and
  $$\psi = \frac{1}{p!q!} \sum \psi_{\lambda_1 \cdots \lambda_p \ov{\mu}_1 \cdots \ov{\mu}_q} dz^{\lambda_1} \wedge \cdots \wedge dz^{\lambda_p} \wedge d\ov{z}^{\mu_1} \wedge \cdots \wedge d\ov{z}^{\mu_q}.$$

We now write down a formula for $*$ in two special cases, which are of most interest to us.  Fix a point on the manifold and suppose that we have chosen complex coordinates $z^1, \ldots, z^n$ in which $g_{i\ov{j}}$ is the identity matrix.  Write
$$e_i = \sqrt{-1} dz^i \wedge d\ov{z}^i, \quad i=1, \ldots, n.$$
 Then for $i=1, \ldots, n$,
 \[
\begin{split}
* e_i = {} & e_1 \wedge \cdots \wedge \widehat{e_i} \wedge \cdots \wedge e_n
\end{split}
\]
and
\[
\begin{split}
* (e_1 \wedge \cdots \wedge \widehat{e_i} \wedge \cdots \wedge e_n) ={} &  e_i,
\end{split}
\]
where we use the symbol \, $\widehat{ \ }\ $ to mean ``omit''. In particular, we see that $*$ maps positive definite real $(1,1)$ forms to positive definite real $(n-1,n-1)$ forms, and vice versa.   Also, observe that if $\chi$ is a real $(1, 1)$ form then
\begin{equation} \label{df2}
\frac{\chi^n}{\omega^n} = \frac{ \det (*\chi)}{ \det (* \omega)}.
\end{equation}

A key formula we need, which can be easily computed using the coordinates given above, is
\begin{equation}
\frac{1}{(n-1)!} * \left( \ddbar u \wedge \omega^{n-2} \right) = \frac{1}{n-1} \left( (\Delta u)\omega - \ddbar u \right),
\end{equation}
for any function $u$. Hence
\begin{equation} \label{corres}
* \frac{1}{(n-1)!}\left( \omega_0^{n-1} + \ddbar u \wedge \omega^{n-2}\right) = \left( \omega_h + \frac{1}{n-1} \left( (\Delta u)\omega - \ddbar u \right) \right)
\end{equation}
where $\omega_h$ is the Hermitian metric defined by
\begin{equation} \label{omegah}
 \omega_h = \frac{1}{(n-1)!}* \omega_0^{n-1}.
\end{equation}
Note that given $\omega_h$ there exists a unique $\omega_0$ solving (\ref{omegah}).

Given (\ref{corres}), the result of Theorem \ref{maintheorem} immediately implies Corollary \ref{cor1}.  Indeed $u$ satisfies the condition (\ref{condthm}) if and only if $u$ satisfies (\ref{condcor1}).  And, using (\ref{df2}),  we see that
\begin{equation}
\begin{split}
\lefteqn{\frac{\left(   \omega_h + \frac{1}{n-1} ((\Delta u)\omega - \ddbar u) \right)^n }{\omega^n}} \\
 = {} & \frac{\det \left( * \left( \omega_h + \frac{1}{n-1} ((\Delta u)\omega - \ddbar u) \right) \right)}{\det (*\omega)} \\
 = {} & \frac{ \det \left(   \omega_0^{n-1} + \ddbar u \wedge \omega^{n-2} \right)}{ \det \left(  \omega^{n-1} \right)},
\end{split}
\end{equation}
so that (\ref{MA1}) holds if and only if (\ref{FWW}) holds.

We now briefly justify Corollary \ref{corbal} using an observation of Fu-Wang-Wu \cite{FWW}.  Let $(u,b)$ solve (\ref{FWW}) with $F = (n-1)F'$.  Then using (\ref{detf}), $\omega_u$ defined by  (\ref{mat}) satisfies
$$e^{F+b}  = \frac{\det (\omega_u^{n-1})}{\det (\omega^{n-1})} = \left( \frac{\omega_u^n}{ \omega^n} \right)^{n-1},$$
and it follows that
$$\omega_u^n = e^{F'+b'} \omega^n,$$
with $b' = b/(n-1)$, which is exactly (\ref{ma}).

We end this section with some final words about notation.  Given any two Hermitian metrics $g$ and $g'$ with associated real $(1,1)$ forms $\omega$ and $\omega'$, we write
$$\tr{\omega}{\omega'} = \tr{g}{g'} = g^{i\ov{j}} g'_{i\ov{j}}.$$
We use letters $C, C'$ etc. to denote uniform constants which depend only on the allowed data, which will be clear from the context.  These constants may differ from line to line.

\section{$L^{\infty}$ estimate}\label{linfty}

In this section we work in the setup of Theorem \ref{maintheorem}.  $(M, \omega)$ is a compact K\"ahler manifold with $\omega_h$ a Hermitian metric.  We obtain an a priori  uniform $L^{\infty}$ estimate for $u$ solving (\ref{MA1}), depending only on the norm $\sup_M |F|$ and the fixed metrics.

 \begin{theorem} \label{theoremC0}  Let $u$ be as in the statement of Theorem \ref{maintheorem}.  Then there exists a constant $C$ depending only on the fixed data $(M, \omega, \omega_h)$ and $\sup_M |F|$ such that
 $$\| u \|_{L^{\infty}}:= \sup_M |u| \le C.$$
 \end{theorem}

First note that the constant $b$ is uniformly bounded in terms of $\sup_M |F|$, $\omega$ and $\omega_h$ by a standard maximum principle argument \cite{Ch, TW1}.  Indeed, at a point at which  $u$ achieves a maximum we have
$$(\Delta u) \omega - \ddbar u \le 0,$$
and from (\ref{MA1})  the upper bound for $b$ follows.  The lower bound is similar.

We introduce some notation that will be used throughout the paper.
Write $$\tilde{F} = F+b.$$
As in Section \ref{prelims}, define  a Hermitian metric $\omega_0$ by $(n-1)! * \omega_h = \omega_0^{n-1}$.  It then follows from the discussion there that
\begin{equation} \label{pd}
\omega_0^{n-1}  + \ddbar u \wedge \omega^{n-2}>0.
\end{equation}
Write
\begin{equation} \label{defntildeo}
\tilde{\omega} = \omega_h + \frac{1}{n-1} (( \Delta u)\omega - \ddbar u).
\end{equation}
Note that taking the trace of this we obtain
$$\tr{\omega}{\tilde{\omega}} = \tr{\omega}{\omega_h} + \Delta u,$$
and hence
\begin{equation} \label{DLB}
\Delta u = \tr{\omega}{\tilde{\omega}} - \tr{\omega}{\omega_h} \ge -C.
\end{equation}

We first prove:
\begin{lemma} \label{keyli}
There exists a uniform constant $C$ such that
\begin{equation} \label{claimo0}
\ddbar u \wedge \left(2 \omega_0^{n-1} + \ddbar u \wedge \omega^{n-2} \right) \le C \omega^n,
\end{equation}
\end{lemma}

\begin{proof}
From (\ref{defntildeo}) and (\ref{DLB}) we have
\begin{equation} \label{ddbu0}
\ddbar u = (n-1) \omega_h + (\tr{\omega}{\tilde{\omega}} - \tr{\omega}{\omega_h} )\omega - (n-1)\tilde{\omega}.
\end{equation}

Compute
\[
\begin{split}
\lefteqn{\ddbar u \wedge \left(2 \omega_0^{n-1} + \ddbar u \wedge \omega^{n-2} \right)} \\
={} & ((n-1) \omega_h + (\tr{\omega}{\tilde{\omega}} - \tr{\omega}{\omega_h} )\omega - (n-1)\tilde{\omega}) \\
&\wedge (2\omega_0^{n-1} + ((n-1) \omega_h + (\tr{\omega}{\tilde{\omega}} - \tr{\omega}{\omega_h} )\omega - (n-1)\tilde{\omega}) \wedge \omega^{n-2})\\
\le {} &  C\omega^n+ (n-1) (\tr{\omega}{\tilde{\omega}}) \omega_h \wedge \omega^{n-1} - (n-1)^2 \omega_h \wedge \tilde{\omega} \wedge \omega^{n-2} \\
& + 2 (\tr{\omega}{\tilde{\omega}}) \omega \wedge \omega_0^{n-1} + (n-1)(\tr{\omega}{\tilde{\omega}}) \omega_h\wedge \omega^{n-1} + (\tr{\omega}{\tilde{\omega}})^2 \omega^n \\
& - (\tr{\omega}{\tilde{\omega}})(\tr{\omega}{\omega_h}) \omega^n - (n-1) (\tr{\omega}{\tilde{\omega}}) \tilde{\omega} \wedge \omega^{n-1} - (\tr{\omega}{\omega_h})(\tr{\omega}{\tilde{\omega}}) \omega^n \\
& + (n-1) (\tr{\omega}{\omega_h}) \tilde{\omega} \wedge \omega^{n-1} - 2(n-1) \tilde{\omega} \wedge \omega_0^{n-1} - (n-1)^2 \tilde{\omega} \wedge \omega_h \wedge \omega^{n-2}\\
& - (n-1) (\tr{\omega}{\tilde{\omega}}) \tilde{\omega} \wedge \omega^{n-1} + (n-1) (\tr{\omega}{\omega_h}) \tilde{\omega} \wedge \omega^{n-1} + (n-1)^2 \tilde{\omega}^2 \wedge \omega^{n-2}.
\end{split}
\]
Then, using the fact that $(\tr{\omega}{\tilde{\omega}}) \omega^n = n \omega^{n-1} \wedge \tilde{\omega}$, we have
\[
\begin{split}
\lefteqn{\ddbar u \wedge \left(2 \omega_0^{n-1} + \ddbar u \wedge \omega^{n-2} \right)} \\
\le {} & C \omega^n + \frac{(n-1)}{n} (\tr{\omega}{\tilde{\omega}}) (\tr{\omega}{\omega_h}) \omega^n - (n-1)^2 \tilde{\omega} \wedge \omega_h \wedge \omega^{n-2} \\
& + \frac{2}{n} (\tr{\omega}{\tilde{\omega}})(\tr{\omega_0}{\omega}) \omega_0^n + \frac{(n-1)}{n} (\tr{\omega}{\tilde{\omega}})(\tr{\omega}{\omega_h}) \omega^n + (\tr{\omega}{\tilde{\omega}})^2 \omega^n \\
& -(\tr{\omega}{\tilde{\omega}}) (\tr{\omega}{\omega_h}) \omega^n -\frac{(n-1)}{n} (\tr{\omega}{\tilde{\omega}})^2 \omega^n - (\tr{\omega}{\omega_h}) (\tr{\omega}{\tilde{\omega}}) \omega^n \\
& + \frac{(n-1)}{n} (\tr{\omega}{\omega_h}) (\tr{\omega}{\tilde{\omega}}) \omega^n - \frac{2(n-1)}{n} (\tr{\omega_0}{\tilde{\omega}}) \omega_0^n - (n-1)^2 \tilde{\omega} \wedge \omega_h \wedge \omega^{n-2} \\
& - \frac{(n-1)}{n} (\tr{\omega}{\tilde{\omega}})^2 \omega^n + \frac{(n-1)}{n} (\tr{\omega}{\omega_h})(\tr{\omega}{\tilde{\omega}}) \omega^n + (n-1)^2 \tilde{\omega}^2 \wedge \omega^{n-2} \\
 = {} & C \omega^n + \frac{(2n-4)}{n} (\tr{\omega}{\tilde{\omega}}) (\tr{\omega}{\omega_h}) \omega^n  - 2(n-1)^2 \tilde{\omega} \wedge \omega_h \wedge \omega^{n-2} \\
& + \frac{2}{n} (\tr{\omega}{\tilde{\omega}}) (\tr{\omega_0}{\omega}) \omega_0^n - \frac{2(n-1)}{n} (\tr{\omega_0}{\tilde{\omega}} ) \omega_0^n - \frac{(n-2)}{n} (\tr{\omega}{\tilde{\omega}})^2 \omega^n \\
& + (n-1)^2 \tilde{\omega}^2 \wedge \omega^{n-2}.
\end{split}
\]
We claim that
\begin{equation} \label{claimmagic}
\begin{split}
\lefteqn{\frac{(2n-4)}{n} (\tr{\omega}{\tilde{\omega}})(\tr{\omega}{\omega_h}) \omega^n - 2(n-1)^2 \tilde{\omega} \wedge \omega_h \wedge \omega^{n-2}} \\
& + \frac{2}{n} (\tr{\omega}{\tilde{\omega}}) (\tr{\omega_0}{\omega}) \omega_0^n - \frac{2(n-1)}{n} (\tr{\omega_0}{\tilde{\omega}}) \omega_0^n =0.
\end{split}
\end{equation}
To prove the claim, we first note that
$$(\tr{\omega_0}{\omega}) \omega_0^n = n \omega_0^{n-1} \wedge \omega = n!  * \omega_h \wedge \omega = n \omega^{n-1} \wedge \omega_h = (\tr{\omega}{\omega_h}) \omega^n,$$
and hence it suffices to prove
\begin{equation} \label{claimmagic2}
(\tr{\omega}{\tilde{\omega}})(\tr{\omega}{\omega_h}) \omega^n - n(n-1) \tilde{\omega} \wedge \omega_h \wedge \omega^{n-2} -  (\tr{\omega_0}{\tilde{\omega}}) \omega_0^n =0.
\end{equation}
Now
 choose coordinates in which, at a  point, $\omega = \sum_{i=1}^n e_i$ and $\omega_0 = \sum_{i=1}^n \alpha_i e_i$, for $e_i$ defined as in Section \ref{prelims}.  Then observe that from the definition of $*$ we have
$$\omega_h = \frac{1}{(n-1)!} * \omega_0^{n-1} =  \sum_{i=1}^n \alpha_1 \cdots \widehat{\alpha_i} \cdots \alpha_n e_i = \frac{\omega_0^n}{\omega^n} \sum_{i=1}^n \frac{1}{\alpha_i} e_i,$$
since
$$\frac{\omega_0^n}{\omega^n} = \alpha_1 \cdots \alpha_n.$$
We write the coefficient of $e_i$ in $\tilde{\omega}$ as $\lambda_i$.  Namely,
$$\tilde{\omega} = \sum_{i=1}^n \lambda_i e_i + \textrm{O.D.T.}$$
where $\textrm{O.D.T.}$ represent ``off-diagonal terms'' containing $dz^k \wedge d\ov{z}^{\ell}$ for $k\neq \ell$, which will disappear in our calculation.  Note that the $\lambda_i$ are positive.
Now compute
\begin{equation} \label{1mag}
\begin{split}
(\tr{\omega}{\tilde{\omega}})(\tr{\omega}{\omega_h}) \omega^n = {} &  \left( \sum_{i,j=1}^n \frac{\lambda_i}{\alpha_j} \right) \omega_0^n,
\end{split}
\end{equation}
and
\begin{equation} \label{2mag}
\begin{split}
\lefteqn{- n(n-1) \tilde{\omega} \wedge \omega_h \wedge \omega^{n-2} } \\= {} & -n(n-1) \left( \sum_{i=1}^n \lambda_i e_i \right) \wedge \left( \sum_{j=1}^n \frac{1}{\alpha_j} e_j \right) \frac{\omega_0^n}{\omega^n}  \\
& \wedge (n-2)! \sum_{k<\ell} e_1 \wedge \cdots \wedge \widehat{e_k} \wedge \cdots \wedge \widehat{e_{\ell}} \wedge \cdots \wedge e_n \\
={} & - n! \left( \sum_{k< \ell} \frac{\lambda_k}{\alpha_{\ell}} + \sum_{k< \ell} \frac{\lambda_{\ell}}{\alpha_k} \right) \frac{\omega_0^n}{\omega^n} e_1\wedge \cdots \wedge e_n\\
={} & -  \left( \sum_{k \neq \ell} \frac{\lambda_k}{\alpha_{\ell}} \right) \omega_0^n,
\end{split}
\end{equation}
and
\begin{equation} \label{3mag}
\begin{split}
-  (\tr{\omega_0}{\tilde{\omega}}) \omega_0^n = - \left( \sum_{i=1}^n \frac{\lambda_i}{\alpha_i} \right) \omega_0^n.
\end{split}
\end{equation}
Combining (\ref{1mag}), (\ref{2mag}) and (\ref{3mag}), we obtain (\ref{claimmagic2}) and this proves the claim (\ref{claimmagic}).

Hence we have shown that
\[
\begin{split}
\lefteqn{\ddbar u \wedge \left(2 \omega_0^{n-1} + \ddbar u \wedge \omega^{n-2} \right)} \\
\le {} &  C \omega^n   - \frac{(n-2)}{n} (\tr{\omega}{\tilde{\omega}})^2 \omega^n
 + (n-1)^2 \tilde{\omega}^2 \wedge \omega^{n-2}.
\end{split}
\]
It suffices to show that the sum of the last two terms on the right hand side is bounded from above.  Choose coordinates so that, at a point, $\omega = \sum_{i=1}^n e_i$ and $\tilde{\omega} = \sum_{i=1}^n \lambda_i e_i$ with  $0<\lambda_1 \le \lambda_2 \le \cdots \le \lambda_n$. Then we have,
 \begin{equation} \nonumber
 \begin{split}
 \lefteqn{- \frac{(n-2)}{n} (\tr{\omega}{\tilde{\omega}})^2 \omega^n
 + (n-1)^2 \tilde{\omega}^2 \wedge \omega^{n-2} } \\
={} & \omega^n \left( - \frac{(n-2)}{n} \left( \sum_{i=1}^n \lambda_i \right)^2  + \frac{2(n-1)}{n} \sum_{i< j} \lambda_i \lambda_j \right) \\
= {} & \frac{\omega^n}{n} \left(  -(n-2) \sum_{i=1}^n \lambda_i^2 - 2(n-2) \sum_{i < j} \lambda_i \lambda_j  + 2(n-1) \sum_{i < j} \lambda_i \lambda_j \right) \\
= {} & \frac{\omega^n}{n} \left( - (n-2) \sum_{i=2}^{n} \lambda_i^2 + 2 \sum_{2 \le i<j \le n} \lambda_i \lambda_j   -(n-2)\lambda_1^2 + 2\lambda_1\left( \sum_{i=2}^n \lambda_i \right)  \right) \\
\le {} &  \frac{\omega^n}{n} \left( - \sum_{2 \le i < j \le n} (\lambda_i - \lambda_j)^2 + 2\lambda_1\left( \sum_{i=2}^n \lambda_i \right)  \right).
 \end{split}
 \end{equation}
To see that this expression is uniformly bounded from above, we make use of the equation
\begin{equation} \label{lambdaeqn}
\lambda_1 \cdots \lambda_n  = e^{\tilde{F}}.
\end{equation}
First suppose that $\lambda_2 < \frac{\lambda_n}{2}$.  Then $(\lambda_2-\lambda_n)^2 \ge \frac{1}{4} \lambda_n^2$.  We immediately obtain
$$ - \sum_{2 \le i < j \le n} (\lambda_i - \lambda_j)^2 + 2\lambda_1\left( \sum_{i=2}^n \lambda_i \right) \le - \frac{1}{4} \lambda_n^2 + C \lambda_n \le C',$$
where we are using the fact that $\lambda_1$ is the smallest eigenvalue and hence is uniformly bounded above.

On the other hand, if $\lambda_2 \ge \lambda_n/2$ then we have $\lambda_i \ge \lambda_n/2$ for $i=2, 3, \ldots, n$.  From (\ref{lambdaeqn}),
$$\lambda_1 \le \frac{C}{\lambda_2 \cdots \lambda_n} \le \frac{C\, 2^{n-2}}{\lambda_n^{n-1}}.$$
 Hence in this case
$$ - \sum_{2 \le i < j \le n} (\lambda_i - \lambda_j)^2 + 2\lambda_1\left( \sum_{i=2}^n \lambda_i \right) \le \frac{C'}{\lambda_n^{n-1}} \lambda_n = \frac{C' }{\lambda_n^{n-2}}\le C',$$
since $\lambda_n$ is the largest eigenvalue and hence is bounded from below uniformly away from zero.
This finishes the proof of the lemma. \end{proof}

Applying this lemma, we obtain the following Cherrier-type estimate (see \cite{Ch} and also \cite{TW2}):

\begin{lemma} \label{Ch} There exists a uniform constant $C$ such that for all $p \ge 1$,
\begin{equation} \label{che}
\int_M | \partial e^{-\frac{pu}{2}} |_g^2 \,  \omega^n \le Cp \int_M e^{-pu} \omega^n.
\end{equation}
\end{lemma}
\begin{proof}
From (\ref{claimo0}), we have
$$\int_M e^{-pu} \ddbar u \wedge (2 \omega_0^{n-1} + \ddbar u \wedge \omega^{n-2} ) \le C \int_M e^{-pu}\omega^n.$$
Integrating by parts,
\[
\begin{split}
\lefteqn{p\int_M e^{-pu} \sqrt{-1} \partial u \wedge \ov{\partial}u \wedge (2 \omega_0^{n-1} + \ddbar u \wedge \omega^{n-2} )} \\
&  - 2\int_M e^{-pu} \sqrt{-1} \partial \omega_0^{n-1} \wedge \ov{\partial} u \le C \int_M e^{-pu}\omega^n.
\end{split}
\]
But from (\ref{pd}), this gives us
\[
\begin{split}
p\int_M e^{-pu} \sqrt{-1} \partial u \wedge \ov{\partial}u \wedge \omega_0^{n-1}
+\frac{2}{p} \int_M  \sqrt{-1} \partial \omega_0^{n-1} \wedge \ov{\partial} (e^{-pu})
 \le C \int_M e^{-pu}\omega^n.
\end{split}
\]
Hence
\[
\begin{split}
\lefteqn{\frac{4}{p} \int_M \sqrt{-1} \partial e^{-\frac{pu}{2}} \wedge \ov{\partial} e^{-\frac{pu}{2}} \wedge \omega_0^{n-1}} \\
 \le {} & C \int_M e^{-pu}\omega^n + \frac{2}{p} \int_M e^{-pu} \sqrt{-1} \partial \ov{\partial} \omega_0^{n-1}
 \le C' \int_M e^{-pu} \omega^n,
\end{split}
\]
and since $\omega_0$ and $\omega$ are uniformly equivalent, this completes the proof of the lemma.
\end{proof}

Given Lemma \ref{Ch}, we can now apply the arguments of \cite{TW1, TW2} to obtain the $L^{\infty}$ bound of $u$.

\begin{proof}[Proof of Theorem \ref{theoremC0}]
It is shown in \cite{TW2} that if (\ref{che}) holds for a uniform $C$ and for all $p$ larger than some uniform constant (here it holds for any $p\ge1$) we obtain the estimate
\begin{equation} \label{size}
| \{ u \le \inf_M u + N \} | \ge \delta,
\end{equation}
for uniform constants $N$ and $\delta>0$.
On the other hand, it is well-known that the conditions $\sup_M u=0$ and $\Delta u \ge -C$ (see (\ref{DLB})) imply that we have a uniform $L^1$ bound for $u$.  Indeed, if $x\in M$ is a point where $u$ achieves its maximum then Green's formula gives us
$$0 = u(x) = \frac{1}{\int_M \omega^n} \int_M u \omega^n - \int_M G(x,y) \Delta u(y) \omega^n(y),$$
where $G(x,y)$ is the Green's function associated to $\Delta$, normalized so that $G(x,y) \ge 0$ and $\int_M G(x,y)\omega^n(y) \le C$.  We conclude that
$$\int_M (-u) \omega^n \le C.$$
The $L^{\infty}$ bound for $u$ now follows immediately from this $L^1$ bound and (\ref{size}).  Indeed,
$$- \delta \inf_M u \le \int_{ \{ u \le \inf_M u+N \}} (-u+N) \le C,$$
giving the required uniform lower bound for $\inf_M u$.
\end{proof}

\section{Second order estimate} \label{soe}

We continue the proof of Theorem \ref{maintheorem} by establishing a uniform a priori estimate for the complex Hessian of $u$
which depends on bounds for the first derivatives of $u$.  We phrase this estimate in terms of the metric $\tilde{g}_{i\ov{j}}$ defined by (\ref{defntildeo}), namely
\begin{equation} \label{tildeg1}
\tilde{g}_{i\ov{j}} = h_{i\ov{j}} + \frac{1}{n-1} \left( (\Delta u)g_{i\ov{j}} -u_{i\ov{j}} \right),
\end{equation}
where we recall that  $$\omega_h = \sqrt{-1} h_{i\ov{j}}dz^i \wedge d\ov{z}^j$$ is a fixed Hermitian metric.
We assume that $u$ satisfies the equation (\ref{MA1}), and from Theorem \ref{theoremC0}, we already have a uniform $L^{\infty}$ bound for $u$.  The estimate is of  the same form as that of Hou-Ma-Wu \cite{HMW} and, as discussed in the introduction, our proof contains many similar elements.

\begin{theorem} \label{theoremC2}
There exists a uniform constant $C$ depending only on $(M, g, h)$ and bounds for $F$ such that
$$\emph{tr}_{\omega}{\tilde{\omega}} \le C( \sup_M |\nabla u|^2_g + 1),$$
where $|\nabla u|^2_g := g^{i\ov{j}} \partial_i u \partial_{\ov{j}} u$.
\end{theorem}
\begin{proof} Recalling (\ref{DLB}),
we define a tensor
\begin{equation} \label{defneta}
\eta_{i\ov{j}} = u_{i\ov{j}} - (n-1) h_{i\ov{j}} + (\tr{g}{h}) g_{i\ov{j}} = (\tr{g}{\tilde{g}})  g_{i\ov{j}} - (n-1) \tilde{g}_{i\ov{j}},
\end{equation}
or in other words,
\begin{equation} \label{defneta2}
\tilde{g}_{i\ov{j}} = \frac{1}{n-1} \left( - \eta_{i\ov{j}} + (\tr{g}{\tilde{g}}) g_{i\ov{j}} \right).
\end{equation}
Following \cite{HMW}, we consider the quantity $H(x, \xi)$ defined by
$$H(x, \xi) = \log(\eta_{i\ov{j}} \xi^i \ov{\xi^j}) + \varphi (|\nabla u|^2_g) + \psi(u),$$
for $x\in M$, $\xi \in T_x^{1,0}M$ a unit vector (with respect to $g$), and for functions $\varphi, \psi$  defined by
\begin{equation}
\begin{split}
\varphi(s) = {} & - \frac{1}{2} \log \left( 1- \frac{s}{2K} \right), \quad \textrm{for } 0 \le s \le K-1, \\
\psi(t) = {} & - A \log \left( 1+ \frac{t}{2L} \right), \quad \textrm{for } -L +1 \le t \le 0,
\end{split}
\end{equation}
where
$$K = \sup_M | \nabla u|^2_g+1,\ L = \sup_M|u| +1, \ A = 2L(C_1 +1)$$
and $C_1$ is a uniformly bounded constant to be determined later.   Recall that we have normalized $u$ so that $\sup_M u=0$. Note that by our $L^{\infty}$ bound for $u$, the quantity $L$ is uniformly bounded.  In addition, $\varphi(|\nabla u|^2_g)$ and $\psi(u)$ are both uniformly bounded.  We remark that, although it may seem more natural to consider $e^H$ instead of $H$,  we prefer the quantity $H$ since, as in \cite{HMW}, it appears to simplify some calculations.

\begin{remark}
In contrast to \cite{HMW}, we introduce the new tensor $\eta_{i\ov{j}}$ in the definition of $H$ instead of working directly with $u_{i\ov{j}}$ because in our coordinates $u_{i\ov{j}}$ is not necessarily diagonal.  However,   $u_{i\ov{j}}$ differs from $\eta_{i\ov{j}}$ by a bounded tensor, and our later calculations will demonstrate that this difference is harmless.
\end{remark}

The function $H$ is not well-defined everywhere.  We restrict $H$ to the compact set $W$ in the $g$-unit tangent bundle of $M$ where $\eta_{i\ov{j}} \xi^i \xi^j \ge 0$, defining $H=-\infty$ when $\eta_{i\ov{j}} \xi^i \xi^j =0$.  Note that $H$ is upper semi-continuous on $W$ and equal to $-\infty$ on the boundary of $W$.   Hence $H$ achieves a maximum at some point $(x_0, \xi_0)$ in the interior of $W$.

We pick a holomorphic coordinate system $z^1, \ldots, z^n$ centered at $x_0$ which is normal for $g$ and with the property that at $x_0$,
$$g_{i\ov{j}} = \delta_{ij}, \quad \eta_{i\ov{j}} = \delta_{ij} \eta_{i\ov{i}},$$
and $$\eta_{1\ov{1}} \ge \eta_{2\ov{2}} \ge \cdots \ge \eta_{n \ov{n}}.$$
It follows from (\ref{defneta2}) that $(\tilde{g}_{i\ov{j}})$ is also diagonal.  Write $\lambda_i = \tilde{g}_{i\ov{i}}$.  Then at $x_0$,
\begin{equation} \label{etalambda}
\eta_{i\ov{i}} = \sum_{j=1}^n \lambda_j - (n-1) \lambda_i,
\end{equation}
and hence
$$0<\lambda_1 \le \lambda_2 \le \cdots \le \lambda_n.$$
The quantities $\lambda_n$, $\eta_{1\ov{1}}$ and $\tr{\omega}{\tilde{\omega}}$ are all uniformly equivalent, and in fact
\begin{equation} \label{ie}
\frac{1}{n} \tr{\omega}{\tilde{\omega}} \le \lambda_n \le \eta_{1\ov{1}} \le (n-1) \lambda_n \le (n-1) \tr{\omega}{\tilde{\omega}}.
\end{equation}

Since $\eta_{1\ov{1}}$ is the largest eigenvalue of $(\eta_{i\ov{j}})$ at $x_0$, we have $\xi_0 = \partial/\partial z^1$ and we extend $\xi_0$ to a locally defined smooth unit vector field
$$\xi_0 = g_{1\ov{1}}^{-1/2} \frac{\partial}{\partial z^1}.$$
Again as in \cite{HMW}, define a new quantity, defined in a neighborhood of $x_0$, by
$$Q(x) = H(x, \xi_0) = \log(g_{1\ov{1}}^{-1} \eta_{1\ov{1}} ) + \varphi (|\nabla u|^2_g) + \psi(u).$$
The function $Q$ has the property that it achieves a maximum at $x=x_0$.  We will show that at $x_0$, we have the bound
$$\eta_{1\ov{1}} \le CK,$$
for a uniform constant $C$, and hence $Q(x_0) \le \log K +C'$.  This will prove the theorem.   Indeed, from (\ref{ie}),
 $\frac{1}{n} \tr{\omega}{\tilde{\omega}}$ is bounded from above by  the largest eigenvalue of $(\eta_{i\ov{j}})$.  Then
 $$\sup_M \tr{\omega}{\tilde{\omega}} \le n \sup_W \eta_{i\ov{j}} \xi^i \ov{\xi^{j}} \le
C e^{H(x_0, \xi_0)} = Ce^{Q(x_0)} \le C'K= C'(\sup_M |\nabla u|^2_g + 1),$$
as required.
Note then that we may assume, without loss of generality, that $\eta_{1\ov{1}}>>1$ and $u_{1\ov{1}}>0$ at $x_0$.

Define a tensor $\Theta^{i\ov{j}}$ by
 $$\Theta^{i\ov{j}} = \frac{1}{n-1} ( (\tr{\tilde{g}}{g})g^{i\ov{j}} - \tilde{g}^{i\ov{j}} )>0,$$ and a linear operator $L$, acting on functions,  by
 $$L(v) = \Theta^{i\ov{j}} v_{i\ov{j}} = \frac{1}{n-1} ( (\tr{\tilde{g}}{g}) \Delta v - \tilde{\Delta} v).$$
 We will compute $L(Q)$ and make use of the fact that at the point $x_0$, we have $L(Q) \le 0$.

Many of the computations from \cite{HMW} go through in a similar way.  As there, we use \emph{covariant derivatives} with respect to the K\"ahler metric  $g$, which we represent using subscripts. Note that since our coordinates are normal for $g$, we have that $\nabla \xi_0(x_0)=0$.  Then at $x_0$,
\begin{equation} \label{Qi}
0 = Q_i = \frac{\eta_{1\ov{1}i}}{\eta_{1\ov{1}}} + \varphi'  \left( \sum_p u_p u_{\ov{p}i} +  \sum_p u_{pi} u_{\ov{p}} \right) + \psi'  u_i,
\end{equation}
where of course $\varphi'$ and $\psi'$ are evaluated at $|\nabla u|^2_g$ and $u$ respectively.  Next, as in \cite{HMW},
\begin{equation}
\begin{split}
Q_{i\ov{j}} = & \frac{\eta_{1\ov{1} i\ov{j}}}{\eta_{1\ov{1}}}- \frac{\eta_{1\ov{1} i} \eta_{1\ov{1} \ov{j}}}{(\eta_{1\ov{1}})^2} +\xi^1_{i\ov{j}}+\ov{\xi^1_{\ov{i}j}} + \psi' u_{i\ov{j}} + \psi'' u_i u_{\ov{j}} \\
& + \varphi'' \left( \sum_p u_p u_{\ov{p} i} + \sum_p u_{pi} u_{\ov{p}} \right) \left( \sum_q u_{q\ov{j}} u_{\ov{q}}  + \sum_q u_q u_{\ov{q} \ov{j}} \right) \\
& + \varphi' \left( \sum_p u_{p\ov{j}} u_{\ov{p} i} + \sum_p u_{pi} u_{\ov{p} \ov{j}} \right) + \varphi' \left( \sum_p u_p u_{\ov{p} i \ov{j}} +  \sum_p u_{pi\ov{j}} u_{\ov{p}}   \right),
\end{split}
\end{equation}
where we write $(\xi^i)$ for the components of $\xi_0$.
Then at $x_0$, $(\Theta^{i\ov{j}})$ is diagonal, and so
\begin{equation}
\begin{split} \label{LQ}
L(Q) = & \sum_i \frac{\Theta^{i\ov{i}} \eta_{1\ov{1} i\ov{i}}}{\eta_{1\ov{1}}} - \sum_i \frac{\Theta^{i\ov{i}} | \eta_{1\ov{1}i}|^2}{(\eta_{1\ov{1}})^2}
+\sum_i\Theta^{i\ov{i}}(\xi^1_{i\ov{i}}+\ov{\xi^1_{\ov{i}i}})+ \psi' \sum_i \Theta^{i\ov{i}} u_{i\ov{i}} \\
&+ \psi'' \sum_i \Theta^{i\ov{i}} |u_i|^2 + \varphi'' \sum_i \Theta^{i\ov{i}} \left| \sum_p u_p u_{\ov{p} i} + \sum_p u_{pi} u_{\ov{p}} \right|^2 \\
&+ \varphi' \sum_{i,p} \Theta^{i\ov{i}} |u_{p \ov{i}}|^2  + \varphi' \sum_{i,p} \Theta^{i\ov{i}} |u_{pi}|^2 + \varphi' \sum_{i,p} \Theta^{i\ov{i}} (u_{pi\ov{i}} u_{\ov{p}} + u_{\ov{p} i \ov{i}} u_p),
\end{split}
\end{equation}
and a direct calculation shows that at $x_0$ we have $\xi^1_{i\ov{i}}=-\ov{\xi^1_{\ov{i}i}}=\frac{1}{2}\de_i\de_{\ov{i}}g_{1\ov{1}}$, and so at $x_0,$
$$\sum_i\Theta^{i\ov{i}}(\xi^1_{i\ov{i}}+\ov{\xi^1_{\ov{i}i}})=0.$$
Next we differentiate (covariantly) the equation
$$\log \frac{ \det \tilde{g} }{\det g} = F+b,$$
and we get
\begin{equation} \label{Fell}
\tilde{g}^{i\ov{j}} \nabla_{\ell} \tilde{g}_{i\ov{j}} = F_{\ell},
\end{equation}
and
$$\tilde{g}^{i\ov{j}} \nabla_{\ov{m}} \nabla_{\ell} \tilde{g}_{i\ov{j}} - \tilde{g}^{i\ov{q}} \tilde{g}^{p\ov{j}} \nabla_{\ov{m}} \tilde{g}_{p\ov{q}} \nabla_{\ell} \tilde{g}_{i\ov{j}} = F_{\ell \ov{m}}.$$
For convenience, define $\hat{h}_{i\ov{j}}=(n-1)h_{i\ov{j}}$.
A short computation shows that the above equation can be written
\begin{equation}
\begin{split}
\lefteqn{\Theta^{i\ov{j}} u_{i \ov{j} \ell \ov{m}} + \tilde{g}^{i\ov{j}} \nabla_{\ov{m}} \nabla_{\ell} h_{i\ov{j}}} \\ & - \frac{1}{(n-1)^2} \tilde{g}^{i\ov{q}} \tilde{g}^{p\ov{j}} ( g_{p\ov{q}} g^{r\ov{s}} u_{r\ov{s} \ov{m}} - u_{p\ov{q} \ov{m}} + \nabla_{\ov{m}} \hat{h}_{p\ov{q}}) (g_{i\ov{j}} g^{a\ov{b}} u_{a\ov{b} \ell} - u_{i\ov{j}\ell}+ \nabla_{\ell} \hat{h}_{i\ov{j}}) =F_{\ell \ov{m}}.
\end{split}
\end{equation}
Hence we get at $x_0$
\begin{equation} \label{F11}
\begin{split}
\lefteqn{
\sum_i \Theta^{i\ov{i}} u_{i\ov{i}  1 \ov{1}} + \sum_i \tilde{g}^{i\ov{i}}  h_{i\ov{i}1\ov{1}}   } \\
 &- \frac{1}{(n-1)^2} \sum_{i,j} \tilde{g}^{i\ov{i}} \tilde{g}^{j\ov{j}} ( g_{j\ov{i}} \sum_a u_{a\ov{a} \ov{1}} - u_{j\ov{i} \ov{1}} + \hat{h}_{j\ov{i} \ov{1}}) (g_{i\ov{j}} \sum_b u_{b\ov{b} 1} - u_{i\ov{j}1} + \hat{h}_{i\ov{j} 1}) =F_{1 \ov{1}}.
 \end{split}
\end{equation}
Commuting covariant derivatives, we obtain (cf. \cite[p. 552]{HMW})
\begin{equation}  \label{commute}
\begin{split}
u_{i\ov{j} \ell} = {} & u_{i\ell \ov{j}} - u_a R_{i \ \ell \ov{j}}^{ \ a},\\
u_{i\ov{j} \ell \ov{m}} = {} & u_{\ell \ov{m} i\ov{j}} + u_{a\ov{j}} R_{i \ \ell \ov{m}}^{ \ a} - u_{a\ov{m}} R_{i \ \ell \ov{j}}^{\ a},
\end{split}
\end{equation}
and hence at the point $x_0$,
$$u_{i\ov{i} 1\ov{1}} = u_{1\ov{1} i\ov{i}} + \sum_a u_{a\ov{i}} R_{i \ov{a} 1 \ov{1}} - \sum_a u_{a\ov{1}} R_{i \ov{a} 1 \ov{i}}.$$
Then from (\ref{F11}) we have
\begin{equation}
\begin{split} \label{F112}
\lefteqn{\sum_i \Theta^{i\ov{i}} u_{ 1 \ov{1} i\ov{i} } +\sum_i \tilde{g}^{i\ov{i}}  h_{i\ov{i}1\ov{1}}  + \sum_{i} \Theta^{i\ov{i}} \left( \sum_a u_{a\ov{i}} R_{i \ov{a} 1 \ov{1}} - \sum_a u_{a\ov{1}} R_{i \ov{a} 1 \ov{i}}\right) } \\
&  - \frac{1}{(n-1)^2} \sum_{i,j} \tilde{g}^{i\ov{i}} \tilde{g}^{j\ov{j}} ( g_{j\ov{i}} \sum_a u_{a\ov{a} \ov{1}} - u_{j\ov{i} \ov{1}}+ \hat{h}_{j\ov{i}\ov{1}}) (g_{i\ov{j}} \sum_b u_{b\ov{b} 1} - u_{i\ov{j}1}+\hat{h}_{i\ov{j}1}) =F_{1 \ov{1}}.
\end{split}
\end{equation}
On the other hand, from the definition of $\eta_{i\ov{j}}$, we have
\begin{equation} \label{u11ii}
u_{1\ov{1} i\ov{i}} = \eta_{1\ov{1}i\ov{i}} +  \hat{h}_{1\ov{1} i\ov{i}} - (\tr{g}{h})_{i\ov{i}}.
\end{equation}

Combining (\ref{LQ}), (\ref{F112}) and (\ref{u11ii}) gives, at $x_0$, since $L(Q)\le 0$,
\begin{equation} \label{firstcombine}
\begin{split}
0 \ge {} & \frac{\sum_{i,j} \tilde{g}^{i\ov{i}} \tilde{g}^{j\ov{j}} ( g_{j\ov{i}} \sum_a u_{a\ov{a} \ov{1}} - u_{j\ov{i} \ov{1}}+ \hat{h}_{j\ov{i}\ov{1}}) (g_{i\ov{j}} \sum_b u_{b\ov{b} 1} - u_{i\ov{j}1} + \hat{h}_{i\ov{j}1})}{{(n-1)^2\eta_{1\ov{1}}}}   - \sum_i \frac{\Theta^{i\ov{i}} | \eta_{1\ov{1}i}|^2}{(\eta_{1\ov{1}})^2} \\
& + \frac{1}{\eta_{1\ov{1}}} \left(  F_{1\ov{1}} - \sum_{i} \Theta^{i\ov{i}} \left( \sum_a u_{a\ov{i}} R_{i \ov{a} 1 \ov{1}} - \sum_a u_{a\ov{1}} R_{i \ov{a} 1 \ov{i}}\right)  \right. \\
& \left. - \sum_i \tilde{g}^{i\ov{i}} h_{i\ov{i}1\ov{1}} - \sum_i \Theta^{i\ov{i}} \left(  \hat{h}_{1\ov{1} i\ov{i}} - (\tr{g}{h})_{i\ov{i}} \right) \right) \\
& + \psi' \sum_i \Theta^{i\ov{i}} u_{i\ov{i}}
+ \psi'' \sum_i \Theta^{i\ov{i}} |u_i|^2 + \varphi'' \sum_i \Theta^{i\ov{i}} \left| \sum_p u_p u_{\ov{p} i} + \sum_p u_{pi} u_{\ov{p}} \right|^2 \\
&+ \varphi' \sum_{i,p} \Theta^{i\ov{i}} |u_{p \ov{i}}|^2  + \varphi' \sum_{i,p} \Theta^{i\ov{i}} |u_{pi}|^2 + \varphi' \sum_{i,p} \Theta^{i\ov{i}} (u_{pi\ov{i}} u_{\ov{p}} + u_{\ov{p} i \ov{i}} u_p).
\end{split}
\end{equation}
Noting that at $x_0$,
\begin{equation} \label{sumT}
\sum_i \Theta^{i\ov{i}} = \tr{\tilde{g}}{g},
\end{equation}
and
$$| u_{i\ov{j}}| \le C \eta_{1\ov{1}}, \quad \textrm{for } i,j, =1, \ldots, n,$$
we have a lower bound for the second and third lines of (\ref{firstcombine}),
\begin{equation} \label{Rlb}
\begin{split}
\lefteqn{\frac{1}{\eta_{1\ov{1}}} \left(  F_{1\ov{1}} - \sum_{i} \Theta^{i\ov{i}}  \left( \sum_a u_{a\ov{i}} R_{i \ov{a} 1 \ov{1}} - \sum_a u_{a\ov{1}} R_{i \ov{a} 1 \ov{i}}\right)  \right. } \\ & \left. - \sum_i \tilde{g}^{i\ov{i}} h_{i\ov{i}1\ov{1}} - \sum_i \Theta^{i\ov{i}} \left( \hat{h}_{1\ov{1} i\ov{i}} - (\tr{g}{h})_{i\ov{i}} \right) \right) \ge  - C \tr{\tilde{g}}{g} - C.
\end{split}
\end{equation}
And
\begin{equation} \label{Thetauii}
\sum_i \Theta^{i\ov{i}} u_{i\ov{i}} = \frac{1}{n-1} ( (\tr{\tilde{g}}{g}) \Delta u - \tilde{\Delta} u) = n - \tr{\tilde{g}}{h},
\end{equation}
where the last equality can be seen by taking the trace with respect to $\tilde{g}$ of (\ref{tildeg1}).

Next, from (\ref{Fell}), a short computation shows that we have
$$\Theta^{i\ov{j}} u_{i\ov{j} \ell} + \tilde{g}^{i\ov{j}} h_{i\ov{j} \ell}= F_{\ell},$$
and at $x_0$ this gives,
$$\sum_i \Theta^{i\ov{i}} u_{i\ov{i} \ell} = F_{\ell} - \sum_i \tilde{g}^{i\ov{i}} h_{i\ov{i}\ell}.$$
Using (\ref{commute}), we get at $x_0$,
$$\sum_i \Theta^{i\ov{i}} u_{p i \ov{i}}= F_{p} - \sum_i \tilde{g}^{i\ov{i}} h_{i\ov{i}p} + \sum_{i, q} \Theta^{i\ov{i}} u_q R_{p \ov{q} i \ov{i}}.$$
On the other hand, $u_{\ov{i}i \ov{p}} = u_{i\ov{i} \ov{p}}= u_{\ov{p} i\ov{i}}$ and so
$$\sum_i \Theta^{i\ov{i}} u_{\ov{p}i\ov{i}} = F_{\ov{p}} - \sum_i \tilde{g}^{i\ov{i}} h_{i\ov{i} \ov{p}}.$$
Hence
\begin{equation}
\begin{split}
\lefteqn{ \varphi' \sum_{i,p} \Theta^{i\ov{i}} (u_{pi\ov{i}} u_{\ov{p}} + u_{\ov{p} i\ov{i}} u_p )  } \\={} & \varphi' \sum_p ( F_p u_{\ov{p}} + F_{\ov{p}} u_p)
- \varphi' \sum_{i,p} \tilde{g}^{i\ov{i}} (h_{i\ov{i}p} u_{\ov{p}} + h_{i\ov{i} \ov{p}} u_p)
\\
& + \varphi' \sum_{i,p,q} \Theta^{i\ov{i}} u_q u_{\ov{p}} R_{p\ov{q} i\ov{i}}  \\
 \ge {} & - (C + | \nabla u|^2_g)  | \varphi'| - C \tr{\tilde{g}}{g} (| \nabla u|^2_g+1)  | \varphi'|.
 \end{split}
\end{equation}
Now, as in \cite{HMW}, we have
\begin{equation} \label{hmwc1}
\begin{split}
&\frac{1}{2K} \ge  \varphi' \ge \frac{1}{4K} >0, \quad  \varphi'' =  2(\varphi')^2 >0
\end{split}
\end{equation}
and, for later use,
\begin{equation} \label{psidp}
\begin{split}
&\frac{A}{L} \ge  - \psi' \ge \frac{A}{2L} = C_1+1, \quad \psi'' \ge  \frac{2\ve}{1-\ve} (\psi')^2, \quad \textrm{for all } \ve \le \frac{1}{2A+1},
\end{split}
\end{equation}
whenever $\vp$ and $\psi$ are evaluated at $|\nabla u|^2_g$ and $u$ respectively (recall that $0 \le |\nabla u|^2_g \le K-1$ and $-L+1\le u \le 0$).
Then
\begin{equation} \label{phiplb}
\varphi' \sum_{i,p} \Theta^{i\ov{i}} (u_{pi\ov{i}} u_{\ov{p}} + u_{\ov{p} i\ov{i}} u_p )   \ge - C - C_0 \tr{\tilde{g}}{g}.
\end{equation}
Combining (\ref{firstcombine}) with (\ref{Rlb}), (\ref{Thetauii}) and (\ref{phiplb}), we obtain
\begin{equation} \label{secondcombine}
\begin{split}
0 \ge {} &  \frac{\sum_{i,j} \tilde{g}^{i\ov{i}} \tilde{g}^{j\ov{j}} ( g_{j\ov{i}} \sum_a u_{a\ov{a} \ov{1}} - u_{j\ov{i} \ov{1}}+ \hat{h}_{j\ov{i}\ov{1}}) (g_{i\ov{j}} \sum_b u_{b\ov{b} 1} - u_{i\ov{j}1} + \hat{h}_{i\ov{j}1})}{{(n-1)^2\eta_{1\ov{1}}}}   \\
&- \sum_i \frac{\Theta^{i\ov{i}} | \eta_{1\ov{1}i}|^2} {(\eta_{1\ov{1}})^2} + \psi'' \sum_i \Theta^{i\ov{i}} |u_i|^2 + \varphi'' \sum_i \Theta^{i\ov{i}} \left| \sum_p u_p u_{\ov{p} i} + \sum_p u_{pi} u_{\ov{p}} \right|^2 \\
& + \varphi' \sum_{i,p} \Theta^{i\ov{i}} |u_{pi}|^2 + \varphi' \sum_i \Theta^{i\ov{i}} u_{i\ov{i}}^2 + \tr{\tilde{g}}{h} ( - \psi' - C_1) - C,
\end{split}
\end{equation}
and this now fixes the constant $C_1$.
Note that we have used the fact that $\tr{\tilde{g}}{h}$ and $\tr{\tilde{g}}{g}$ are uniformly equivalent.
We define
$$\delta = \frac{1}{1+2A}= \frac{1}{1+ 4L(C_1+1)},$$
and notice that $\delta$ is a uniform constant since we know that $L$ and $C_1$ are bounded.
We consider two cases, which are analogous to the two cases in \cite{HMW}.

\bigskip
\noindent
{\bf Case 1. \ Assume $\lambda_2 \le (1-\delta) \lambda_n$}.  We make use of (\ref{Qi}) and (\ref{sumT}) to see that
\begin{equation}
\begin{split} \label{ft}
- \sum_i \frac{\Theta^{i\ov{i}} | \eta_{1\ov{1}i}|^2}{(\eta_{1\ov{1}})^2} = {} & - \sum_i \Theta^{i\ov{i}} \left| \varphi'  \left( \sum_p u_p u_{\ov{p}i} +  \sum_p u_{pi} u_{\ov{p}} \right) + \psi'  u_i \right|^2 \\
\ge {} &  - 2 (\varphi')^2 \sum_i \Theta^{i\ov{i}} \left| \sum_p u_p u_{\ov{p}i} +  \sum_p u_{pi} u_{\ov{p}} \right|^2 - 2 (\psi')^2 K \tr{\tilde{g}}{g} \\
\ge {} &  - 2 (\varphi')^2 \sum_i \Theta^{i\ov{i}} \left| \sum_p u_p u_{\ov{p}i} +  \sum_p u_{pi} u_{\ov{p}} \right|^2 - 8 (C_1+1)^2 K\tr{\tilde{g}}{g},
\end{split}
\end{equation}
since $| \psi'| \le A/L = 2(C_1+1)$.  Then, using   $\varphi'' = 2(\varphi')^2$ and $\psi''>0$ and $\varphi'>0$ and $ -\psi' - C_1 \ge 1$ we obtain from (\ref{secondcombine}),
\begin{equation}
\begin{split}
0 \ge {} & \varphi' \sum_i \Theta^{i\ov{i}} u_{i\ov{i}}^2 - 8(C_1+1)^2 K \tr{\tilde{g}}{g} - C,
\end{split}
\end{equation}
but since $\varphi' \ge \frac{1}{4K}$ we have
\begin{equation}
\begin{split}
0 \ge {} & \frac{1}{4K} \sum_i \Theta^{i\ov{i}} u_{i\ov{i}}^2 - 8(C_1+1)^2 K \tr{\tilde{g}}{g} - C,
\end{split}
\end{equation}
giving
\begin{equation}
\begin{split}
 \frac{1}{4K} \Theta^{n\ov{n}} u_{n\ov{n}}^2 \le 8(C_1+1)^2 K \tr{\tilde{g}}{g} + C.
\end{split}
\end{equation}

Now by the definition of $\Theta^{i\ov{i}}$ and the fact that $\tilde{g}_{n\ov{n}} \ge \cdots \ge \tilde{g}_{1\ov{1}}$ at $x_0$, we see that
$$\Theta^{n\ov{n}} \ge \cdots \ge \Theta^{1\ov{1}}>0.$$
Then it follows that $$\Theta^{n\ov{n}} \ge \frac{1}{n} \sum_i \Theta^{i\ov{i}} = \frac{1}{n} \tr{\tilde{g}}{g}.$$
Hence
\begin{equation} \label{goodone}
\frac{1}{n} (\tr{\tilde{g}}{g}) u_{n\ov{n}}^2 \le 32(C_1+1)^2 K^2 \tr{\tilde{g}}{g} + CK.
\end{equation}
But the assumption $\lambda_2 \le (1-\delta) \lambda_n$ together with (\ref{defneta}), implies that
\[
\begin{split}
u_{n\ov{n}} \le {} & \sum_{i=1}^n \lambda_i - (n-1) \lambda_n + C \\
\le {} & \lambda_1+\lambda_2+(n-2)\lambda_n - (n-1) \lambda_n + C\\
\le {} & \lambda_1 - \delta \lambda_n +C  \\
\le {} & - \delta \lambda_n +C'\\
\le {} & -\frac{\delta}{2} \lambda_n,
\end{split}
\]
where we use the following: since $\lambda_1 \le \cdots \le \lambda_n$ and the equation gives us $\lambda_1 \cdots \lambda_n = e^{\tilde{F}}$ (recall that $\tilde{F}=F+b$ is bounded), we may assume without loss of generality that $\lambda_1 << 1$, say, and $\lambda_n$ is large compared to $C'/\delta$.

Hence we have $u_{n\ov{n}}^2 \ge \frac{\delta^2}{4} \lambda_n^2$
which from (\ref{goodone}) gives the uniform bound
$$\lambda_n \le C K.$$
By (\ref{ie}), this  implies the desired estimate $\eta_{1\ov{1}} \le C'K$.

\bigskip
\noindent
{\bf Case 2. \ Assume $\lambda_2 \ge (1-\delta) \lambda_n$}.
First we look at the summand $i=1$  in the bad term $- \sum_i \frac{\Theta^{i\ov{i}} | \eta_{1\ov{1}i}|^2} {(\eta_{1\ov{1}})^2} $ in  (\ref{secondcombine}).  We compute as in (\ref{ft}),
\begin{equation}
\begin{split}
- \frac{\Theta^{1\ov{1}} | \eta_{1\ov{1} 1}|^2}{\eta_{1\ov{1}}} \ge {} & - 2 (\varphi')^2  \Theta^{1\ov{1}} \bigg| \sum_p u_p u_{\ov{p} 1} + \sum_p u_{p1} u_{\ov{p}} \bigg|^2 - 2 (\psi')^2  \Theta^{1\ov{1}} |u_1|^2 \\
\ge {} & - 2 (\varphi')^2  \Theta^{1\ov{1}} \bigg|  \sum_p u_p u_{\ov{p} 1}+ \sum_p u_{p1} u_{\ov{p}} \bigg|^2 - 8(C_1+1)^2 n  K \Theta^{1\ov{1}},
\end{split}
\end{equation}
Then, using (\ref{hmwc1}) and (\ref{psidp}), and in particular the inequality $-\psi' - C_1\ge 1$ we obtain from (\ref{secondcombine}),
\begin{equation} \label{ft2}
\begin{split}
0 \ge {} & \frac{\sum_{i,j} \tilde{g}^{i\ov{i}} \tilde{g}^{j\ov{j}} ( g_{j\ov{i}} \sum_a u_{a\ov{a} \ov{1}} - u_{j\ov{i} \ov{1}} + \hat{h}_{j\ov{i} \ov{1}}) (g_{i\ov{j}} \sum_b u_{b\ov{b} 1} - u_{i\ov{j}1} + \hat{h}_{i\ov{j}1} )}{{(n-1)^2\eta_{1\ov{1}}}}    \\
&{} - \sum_{i =2}^n \frac{\Theta^{i\ov{i}} | \eta_{1\ov{1}i}|^2}{(\eta_{1\ov{1}})^2}+ \psi'' \sum_i \Theta^{i\ov{i}} |u_i|^2 + \varphi'' \sum_{i =2}^n \Theta^{i\ov{i}} \left| \sum_p u_p u_{\ov{p} i} + \sum_p u_{pi} u_{\ov{p}} \right|^2 \\
& + \varphi' \sum_{i,p} \Theta^{i\ov{i}} |u_{pi}|^2 + \frac{1}{4K} \sum_i \Theta^{i\ov{i}} u_{i\ov{i}}^2 + \frac{1}{C_0}\tr{\tilde{g}}{g}
 - 8(C_1+1)^2 n  K \Theta^{1\ov{1}} -C,
\end{split}
\end{equation}
for a uniform $C_0>0$.
Now without loss of generality, we may assume that
\begin{equation} \label{ea}
\frac{1}{4K} \sum_i \Theta^{i\ov{i}} u_{i\ov{i}}^2 \ge 8(C_1+1)^2 n  K \Theta^{1\ov{1}},
\end{equation}
since if not then we get the upper bound $u_{1\ov{1}} \le CK$, which implies from (\ref{defneta})  the uniform bound $\eta_{1\ov{1}} \le C K$ that we want.
The following is a key lemma:

\begin{lemma} \label{klemma} We have, at $x_0$,
\begin{equation} \nonumber
\begin{split}
& \frac{\sum_{i,j} \tilde{g}^{i\ov{i}} \tilde{g}^{j\ov{j}} ( g_{j\ov{i}} \sum_a u_{a\ov{a} \ov{1}} - u_{j\ov{i} \ov{1}}+ \hat{h}_{j\ov{i} \ov{1}}) (g_{i\ov{j}} \sum_b u_{b\ov{b} 1} - u_{i\ov{j}1} + \hat{h}_{i\ov{j}1} )}{{(n-1)^2 \eta_{1\ov{1}}}}   - \sum_{i =2}^n \frac{\Theta^{i\ov{i}} | \eta_{1\ov{1}i}|^2}{(\eta_{1\ov{1}})^2} \\
&{} + \psi'' \sum_i \Theta^{i\ov{i}} |u_i|^2 + \varphi'' \sum_{i =2}^n \Theta^{i\ov{i}} \left| \sum_p u_p u_{\ov{p} i} + \sum_p u_{pi} u_{\ov{p}} \right|^2   \ge - \frac{1}{2C_0} \emph{tr}_{\tilde{g}}{g}.
\end{split}
\end{equation}
\end{lemma}

Given the lemma we're done: indeed just combine (\ref{ft2}), (\ref{ea}) with the result of the lemma and we obtain at $x_0$,
$$0 \ge \frac{1}{2C_0}\tr{\tilde{g}}{g} - C$$
and so  $\tr{\tilde{g}}{g} \le C$ at this point.  Hence $\tilde{g}$ is bounded from above and below at this point (since $\tilde{g}$ has bounded determinant).  Then $\eta_{1\ov{1}}$ is uniformly bounded from above by (\ref{ie}) and we're done.

\begin{proof}[Proof of Lemma \ref{klemma}]  First, observe that applying again (\ref{Qi}) and the fact that $\varphi'' = 2(\varphi')^2$,
\begin{equation} \label{uno}
\begin{split}
\varphi'' \sum_{i =2}^n \Theta^{i\ov{i}} \left| \sum_p u_p u_{\ov{p} i} + \sum_p u_{pi} u_{\ov{p}} \right|^2 = {} & 2 \sum_{i =2}^n \Theta^{i\ov{i}} \left| \frac{\eta_{1\ov{1} i}}{\eta_{1\ov{1}}} + \psi' u_i \right|^2 \\
\ge {} & 2 \delta \sum_{i = 2}^n \frac{\Theta^{i\ov{i}} |\eta_{1\ov{1}i}|^2}{(\eta_{1\ov{1}})^2} - \frac{2\delta (\psi')^2}{1-\delta} \sum_{i = 2}^n \Theta^{i\ov{i}} |u_i|^2,
\end{split}
\end{equation}
where for the last line we have used the elementary proposition (Proposition 2.3 in \cite{HMW}) that for any $a, b \in \mathbb{C}^n$ and any $0<\delta'<1$ we have
$$|a+b|^2 \ge \delta' |a|^2 - \frac{\delta'}{1-\delta'} |b|^2.$$
But recall that we defined $\delta = \frac{1}{1+2A}$ and hence from (\ref{psidp}) we have
\begin{equation} \label{dos}
 \frac{2\delta (\psi')^2}{1-\delta} \sum_{i = 2}^n \Theta^{i\ov{i}} |u_i|^2 \le \psi'' \sum_{i = 2}^n \Theta^{i\ov{i}} |u_i|^2 \le \psi'' \sum_{i=1}^n \Theta^{i\ov{i}} |u_i|^2.
 \end{equation}
Combining (\ref{uno}) and (\ref{dos}) we have
\begin{equation} \nonumber
\begin{split}
& \frac{\sum_{i,j} \tilde{g}^{i\ov{i}} \tilde{g}^{j\ov{j}} ( g_{j\ov{i}} \sum_a u_{a\ov{a} \ov{1}} - u_{j\ov{i} \ov{1}}+ \hat{h}_{j\ov{i} \ov{1}}) (g_{i\ov{j}} \sum_b u_{b\ov{b} 1} - u_{i\ov{j}1} + \hat{h}_{i\ov{j}1})}{{(n-1)^2\eta_{1\ov{1}}}}  \\ &{} - \sum_{i =2}^n \frac{\Theta^{i\ov{i}} | \eta_{1\ov{1}i}|^2}{(\eta_{1\ov{1}})^2}
 + \psi'' \sum_i \Theta^{i\ov{i}} |u_i|^2 + \varphi'' \sum_{i =2}^n \Theta^{i\ov{i}} \left| \sum_p u_p u_{\ov{p} i} + \sum_p u_{pi} u_{\ov{p}} \right|^2  \\
  \ge {} & \frac{\sum_{i,j} \tilde{g}^{i\ov{i}} \tilde{g}^{j\ov{j}} ( g_{j\ov{i}} \sum_a u_{a\ov{a} \ov{1}} - u_{j\ov{i} \ov{1}} + \hat{h}_{j\ov{i} \ov{1}} ) (g_{i\ov{j}} \sum_b u_{b\ov{b} 1} - u_{i\ov{j}1} + \hat{h}_{i\ov{j} 1})}{{(n-1)^2\eta_{1\ov{1}}}} \\
& {} - (1-2\delta) \sum_{i =2}^n \frac{\Theta^{i\ov{i}} | \eta_{1\ov{1}i}|^2}{(\eta_{1\ov{1}})^2}.
\end{split}
\end{equation}
Hence it is sufficient to show that, at a maximum point of $Q$, the right hand side of the above inequality is bounded below by $-\frac{1}{2C_0} \tr{\tilde{g}}{g}$.
In fact, we will show the stronger inequality
\begin{equation} \label{keyi2}
\begin{split}
\lefteqn{\frac{\sum_{i=2}^n \tilde{g}^{i\ov{i}} \tilde{g}^{1\ov{1}} ( g_{1\ov{i}} \sum_a u_{a\ov{a} \ov{1}} - u_{1\ov{i} \ov{1}}+ \hat{h}_{1\ov{i} \ov{1}} ) (g_{i\ov{1}} \sum_b u_{b\ov{b} 1} - u_{i\ov{1}1} + \hat{h}_{i\ov{1}1})}{{(n-1)^2\eta_{1\ov{1}}}} } \qquad \qquad \qquad \qquad \\
& - (1-2\delta) \sum_{i =2}^n \frac{\Theta^{i\ov{i}} | \eta_{1\ov{1}i}|^2}{(\eta_{1\ov{1}})^2} \ge  - \frac{1}{2C_0} \tr{\tilde{g}}{g}. \qquad \qquad
\end{split}
\end{equation}
First note that $u_{1\ov{i} \ov{1}}= u_{1 \ov{1} \ov{i}}$ and $u_{i\ov{1}1} = u_{\ov{1}1 i}$. Furthermore, at $x_0$ we have $g_{1\ov{i}}=0$ for $i=2, 3, \ldots, n$. Then
\begin{equation} \label{keyi3}
\begin{split}
\frac{\sum_{i=2}^n \tilde{g}^{i\ov{i}} \tilde{g}^{1\ov{1}} ( g_{1\ov{i}} \sum_a u_{a\ov{a} \ov{1}} - u_{1\ov{i} \ov{1}}+ \hat{h}_{1\ov{i} \ov{1}} ) (g_{i\ov{1}} \sum_b u_{b\ov{b} 1} - u_{i\ov{1}1} + \hat{h}_{i\ov{1}1})}{{(n-1)^2\eta_{1\ov{1}}}}  \\
=  \frac{\sum_{i=2}^n \tilde{g}^{i\ov{i}} \tilde{g}^{1\ov{1}} |  u_{1\ov{1} \ov{i}}- \hat{h}_{1\ov{i} \ov{1}}|^2}{{(n-1)^2\eta_{1\ov{1}}}}. \qquad \qquad \qquad \qquad
\end{split}
\end{equation}
Next
from the definition of $\eta_{i\ov{j}}$ we see that we can write
$$u_{1\ov{1}\ov{i}} - \hat{h}_{1\ov{i}\ov{1}} = \eta_{1\ov{1}\ov{i}} + e_{1\ov{1}\ov{i}},$$
where $e_{i\ov{j} \ov{k}}$ are the components of a uniformly bounded tensor.
Hence
\begin{equation} \label{keyi4}
\begin{split}
\lefteqn{\frac{\sum_{i=2}^n \tilde{g}^{i\ov{i}} \tilde{g}^{1\ov{1}} |  u_{1\ov{1} \ov{i}}- \hat{h}_{1\ov{i} \ov{1}}|^2}{{(n-1)^2\eta_{1\ov{1}}}} } \\
\ge {} & (1-\delta/2) \frac{\sum_{i=2}^n \tilde{g}^{i\ov{i}} \tilde{g}^{1\ov{1}} | \eta_{1\ov{1} \ov{i}}|^2}{{(n-1)^2\eta_{1\ov{1}}}}
 - C_{\delta}  \frac{\sum_{i=2}^n  \tilde{g}^{i\ov{i}} \tilde{g}^{1\ov{1}} |e_{1\ov{1} \ov{i}}|^2 }{{(n-1)^2\eta_{1\ov{1}}}},
\end{split}
\end{equation}
for a constant $C_{\delta}$ depending only on $\delta$.

On the other hand, we have $\lambda_2 \ge (1-\delta)\lambda_n$ and hence $\tilde{g}^{i\ov{i}}\le \frac{1}{(1-\delta) \lambda_n}$ for $i=2, 3, \ldots, n$.  Since we may assume without loss of generality that $\lambda_n$ is large compared to $C_{\delta}$, we have
\begin{equation} \label{ebd}
C_{\delta}  \frac{\sum_{i=2}^n  \tilde{g}^{i\ov{i}} \tilde{g}^{1\ov{1}} |e_{1\ov{1} \ov{i}}|^2 }{{(n-1)^2\eta_{1\ov{1}}}} \le \frac{1}{2C_0} \tr{\tilde{g}}{g}.
\end{equation}
Then combining (\ref{keyi3}), (\ref{keyi4}) and (\ref{ebd}), to prove (\ref{keyi2}) it suffices to show that
\begin{equation}
\begin{split}
(1-\delta/2) \frac{\sum_{i=2}^n \tilde{g}^{i\ov{i}} \tilde{g}^{1\ov{1}}  | \eta_{1\ov{1} \ov{i}}|^2}{{(n-1)^2\eta_{1\ov{1}}}}  \ge (1- 2\delta) \sum_{i =2}^n \frac{\Theta^{i\ov{i}} |\eta_{1\ov{1}i}|^2}{(\eta_{1\ov{1}})^2},
 \end{split}
\end{equation}
and it is enough to prove
\begin{equation} \label{keyiii}
\begin{split}
\frac{\sum_{i=2}^n \tilde{g}^{i\ov{i}} \tilde{g}^{1\ov{1}}  | \eta_{1\ov{1} \ov{i}}|^2}{{(n-1)^2 \eta_{1\ov{1}} }}  \ge \left(1- \frac{3}{2}\delta\right) \sum_{i =2}^n \frac{\Theta^{i\ov{i}} |\eta_{1\ov{1}i}|^2}{(\eta_{1\ov{1}})^2},
 \end{split}
\end{equation}

The rest of the proof consists of proving this inequality (\ref{keyiii}).
In fact, we will prove for each $i=2, 3, \ldots, n$,
\begin{equation} \label{stronger1}
\left(1- \frac{3}{2}\delta\right) \frac{\Theta^{i\ov{i}}}{\eta_{1\ov{1}}} \le \frac{1}{(n-1)^2} \tilde{g}^{i\ov{i}} \tilde{g}^{1\ov{1}},
\end{equation}
which implies (\ref{keyiii}).

Phrasing (\ref{stronger1}) in terms of the $\lambda_i$, making use of (\ref{etalambda}), we see that it suffices to prove, for $i=2, \ldots, n$,
\begin{equation} \nonumber
\frac{1-3\delta/2}{\lambda_2+\cdots + \lambda_n - (n-2)\lambda_1}  \sum_{j \neq i} \frac{1}{\lambda_j} \le \frac{1}{(n-1) \lambda_i \lambda_1}.
\end{equation}
But notice that since $\lambda_n \ge \lambda_{n-1} \ge \cdots \ge \lambda_1>0$, it's enough to prove this for $i=n$, namely
\begin{equation} \label{suffices1}
\frac{1-3\delta/2}{\lambda_2+\cdots + \lambda_n - (n-2)\lambda_1}  \left( \frac{1}{\lambda_1} + \sum_{j=2}^{n-1} \frac{1}{\lambda_j} \right) \le \frac{1}{(n-1) \lambda_n \lambda_1}.
\end{equation}
From the assumption $\lambda_2 \ge (1-\delta)\lambda_n$, the eigenvalues  $\lambda_2, \ldots, \lambda_n$ are all large compared to $\lambda_1$ and so we may assume without loss of generality that
\begin{equation} \label{2assume21}
\frac{1}{\lambda_1} + \sum_{j=2}^{n-1} \frac{1}{\lambda_j} \le (1+\ve) \frac{1}{\lambda_1},
\end{equation}
for a  small $\ve>0$, depending only on $\delta$, which will be chosen later.
Next, again from the assumption that $\lambda_2 \ge (1-\delta)\lambda_n$, and assuming without loss of generality that $(n-2) \lambda_1 \le \delta \lambda_n$,
we have
\begin{equation} \label{claimlambda1}
\begin{split}
\lambda_2+\cdots + \lambda_n - (n-2) \lambda_1 \ge {} & (n-2) (1-\delta) \lambda_n + \lambda_n - \delta \lambda_n \\
= {} & (n-1)(1-\delta) \lambda_n.
\end{split}
\end{equation}
From (\ref{2assume21}), (\ref{claimlambda1}) we see that to prove (\ref{suffices1}), it suffices to show
\begin{equation}
\frac{(1-3\delta/2)(1+\ve)}{1-\delta} \le 1.
\end{equation}
But we can arrange this by choosing
$$\ve = \frac{\delta}{2-3\delta}>0.$$
This proves (\ref{suffices1}) and hence (\ref{keyiii}), completing the proof of the lemma.
\end{proof}

This completes the proof of Theorem \ref{theoremC2}.
\end{proof}

\begin{remark} It would be interesting to know whether $\tr{\omega}{\tilde{\omega}}$ can be directly bounded in terms of $u$.  By comparison, in the case of the usual complex Monge-Amp\`ere equation on Hermitian manifolds the estimate
$$\tr{\omega}{\tilde{\omega}} \le C e^{A(u - \inf_M u)}$$
was proved in \cite{TW3} by  adapting a trick of Phong-Sturm \cite{PS, PS2}.
\end{remark}

\section{First order estimate} \label{foe}

In this section we prove the following result.

\begin{theorem}\label{grad}
In the setting of Theorem \ref{maintheorem}, let $u$ solve the equation (\ref{MA1}).
Then there is a constant $C$ which depends only on
$\|F\|_{C^2(M,g)}$ (as well as the fixed data $M,\omega,\omega_h$), such that
\begin{equation}\label{gradest}
\sup_M |\nabla u|_g\leq C.
\end{equation}
\end{theorem}

Using a blow-up argument we will prove this result by establishing a Liouville-type theorem, following the proof by Dinew-Ko\l odziej \cite{DK} of an analogous result for the complex Hessian equations.  Note that our argument here can be carried over to the more general  case of equation (\ref{eco}) for Hermitian  $\omega$.

First, we need some notation.
Let $\Omega\subset\mathbb{C}^n, n\geq 2,$ be a domain (possibly the whole of $\mathbb{C}^n$) and $u:\Omega\to\mathbb{R}\cup\{-\infty\}$ be an upper semicontinuous function which is in $L^1_{\mathrm{loc}}(\Omega)$. If
$$P(u):=\frac{1}{n-1}\left((\Delta u)\beta-\ddbar u\right)\geq 0,$$
as a real $(1,1)$ current, we will say that $u$ is $(n-1)$-PSH. Here $\beta$ is the Euclidean K\"ahler form on $\mathbb{C}^n$ and $\Delta u$ is its Laplacian.
Taking the trace of $P(u)\geq 0$ with respect to $\beta$, we see in particular that $u$ is subharmonic.

Harvey-Lawson (see e.g. \cite[Theorem 9.2]{HL}) have shown that, up to modifying $u$ on a set of measure zero, this is equivalent to requiring that given any affine linear $(n-1)$-dimensional complex subspace $H\subset\mathbb{C}^n$, we have that
$u|_{\Omega\cap H}$ is subharmonic.
In this section, we will always assume that our $(n-1)$-PSH functions are continuous (and with values in $\mathbb{R}$). This is enough for our purposes, and it is not much harder to remove this assumption and develop the basic theory below.

In contrast with the case of the complex Monge-Amp\`ere operator \cite{BT}, it is not clear in general how to define
$P(u)^n$ in the sense of pluripotential theory. However, there is a suitable substitute for the condition that $P(u)^n=0$, which is preserved under uniform limits, namely the notion of maximality.

Let us say that a continuous $(n-1)$-PSH function $u$ is \emph{maximal} if given any relatively compact open set $\Omega'\Subset \Omega$ and any continuous $(n-1)$-PSH function
$v$ on a domain $\Omega''$ with $\Omega'\Subset\Omega''\Subset\Omega$ and with $v\leq u$ on $\de\Omega'$, then we must have that $v\leq u$ on $\Omega'$.

Note that if $u$ is a continuous $(n-1)$-PSH and $\chi_\ve$ is a family of mollifiers, we can define $u_\ve=u*\chi_\ve$. These are functions in $C^\infty(\Omega_\ve)$, where $\Omega_\ve=\{x\in \Omega\ |\ d(x,\de\Omega)>\ve\}$, which are also $(n-1)$-PSH since $P(u_\ve)=P(u)*\chi_\ve\geq 0$, and $u_\ve\to u$ locally uniformly as $\ve\to 0$.

The key result that we need is the following:
\begin{theorem}\label{liouv}
If $u:\mathbb{C}^n\to \mathbb{R}$ is an $(n-1)$-PSH function in $\mathbb{C}^n$ which is Lipschitz continuous,
maximal and satisfies $$\sup_{\mathbb{C}^n}(|u|+|\nabla u|)<\infty,$$
then $u$ is a constant.
\end{theorem}
This is completely analogous to the Liouville theorem of Dinew-Ko\l odziej \cite{DK} in the context of complex Hessian equations.
Recall that Lipschitz continuous functions are differentiable a.e., so the bound $\sup_{\mathbb{C}^n}|\nabla u|<\infty$ is meant as an essential supremum, or equivalently as a bound on the global Lipschitz constant of $u$.

We begin with a comparison principle result for $(n-1)$-PSH functions.

\begin{lemma}\label{comparison}
Let $\Omega\subset\mathbb{C}^n$ be a bounded domain and $u,v: \ov{\Omega} \rightarrow \mathbb{R}$ continuous functions, smooth on $\Omega$, with $P(u)\geq 0, P(v)\geq 0$, such that $P(v)^n\geq P(u)^n$ on $\Omega$ and $v\leq u$ on $\de\Omega$. Then $v\leq u$ on $\Omega$.
\end{lemma}
\begin{proof}
For $\ve>0$ let
$$v_\ve(z)=v(z)+\ve(|z|^2-\sup_{w\in\de\Omega}|w|^2),$$
so that $P(v_\ve)=P(v)+\ve\beta>0$ and $P(v_\ve)^n>P(u)^n$ in $\Omega$, and $v_\ve\leq u$ on $\de\Omega$.
Then on $\Omega$ we have
\[\begin{split}
0>P(u)^n-P(v_\ve)^n&=\int_0^1 \frac{d}{dt} P(tu+(1-t)v_\ve)^n dt\\
&=\int_0^1 n P(tu+(1-t)v_\ve)^{n-1}\wedge P(u-v_\ve)dt\\
&=A^{i\ov{j}}\frac{\de^2}{\de z^i \de\ov{z}^j} (u-v_\ve),
\end{split}\]
where $A^{i\ov{j}}\frac{\de^2}{\de z^i \de\ov{z}^j}$ is obtained by taking the linearization  of
$P(u)^n$ at $tu+(1-t)v_\ve$ and integrating on $0\leq t\leq 1$, and the Hermitian matrix $A^{i\ov{j}}(z)$ is positive definite everywhere on $\Omega$.
This inequality implies that the minimum of $u-v_\ve$ on $\ov{\Omega}$ is achieved on $\de\Omega$, hence $v_\ve\leq u$ on $\Omega$. Letting $\ve\to 0$
gives the desired conclusion.
\end{proof}
Assuming Theorem \ref{liouv}, we can now prove Theorem \ref{grad}:
\begin{proof}[Proof of Theorem \ref{grad}]
Assume for a contradiction that \eqref{gradest} does not hold. Then there exist a sequence of smooth functions $F_j$ with $\|F_j\|_{C^2(M,g)}\leq C$,
and smooth functions $u_j$ with $\sup_M u_j=0$ and
$$\ti{\omega}_j:=\omega_h+\frac{1}{n-1}((\Delta_g u_j)\omega-\ddbar u_j)>0,$$
solving the equations
$$\ti{\omega}_j^n=e^{F_j}\omega^n, \quad \textrm{with }
C_j:=\sup_M |\nabla u_j|_g\to\infty.$$
From Theorem \ref{theoremC0} we know that $\sup_M|u_j|\leq C$.
For each $j$ let $x_j\in M$ be a point where the maximum of $|\nabla u_j|_g$ is achieved. Up to passing to a subsequence, we can assume that
$x_j\to x\in M$. Fix a coordinate chart centered at $x$, which we identify with the ball $B_2(0)$ of radius 2 centered at the origin in $\mathbb{C}^n$ with coordinates $(z^1,\dots, z^n)$,
and such that $\omega(x)=\beta$, the identity.
From now on assume that $j$ is sufficiently large so that the points $x_j$ are contained in $B_1(0)$. We define, on the ball $B_{C_j}(0)$ in $\mathbb{C}^n$,
$$\hat{u}_j(z)=u_j(C_j^{-1}z + x_j).$$
The function $\hat{u}_j$ satisfies
$$\sup_{B_{C_j}(0)}|\hat{u}_j|\leq C, \quad \textrm{and }
\sup_{B_{C_j}(0)}|\nabla\hat{u}_j|\leq C,$$
where here the gradient is the Euclidean gradient. Furthermore,
$$|\nabla \hat{u}_j|(0)=C_j^{-1}|\nabla u_j|_g(x_j)=1.$$
Thanks to Theorem \ref{theoremC2}, we also have
that
\begin{equation}\label{c2e}
\sup_{B_{C_j}(0)} |\ddbar \hat{u}_j |_{\beta}\leq CC_j^{-2}\sup_M  |\ddbar u_j |_{\omega}\leq C'.
\end{equation}
Using the elliptic estimates for $\Delta$ and the Sobolev embedding, we see that for each given $K\subset\mathbb{C}^n$ compact, each $0<\alpha<1$ and $p>1$, there is a constant $C$ such
that $$\|\hat{u}_j\|_{C^{1,\alpha}(K)}+\|\hat{u}_j\|_{W^{2,p}(K)}\leq C.$$
Therefore  a subsequence of $\hat{u}_j$ converges strongly in $C^{1,\alpha}_{\mathrm{loc}}(\mathbb{C}^n)$ as well
as weakly in $W^{2,p}_{\mathrm{loc}}(\mathbb{C}^n)$ to a function $u\in W^{2,p}_{\mathrm{loc}}(\mathbb{C}^n)$ with
$\sup_{\mathbb{C}^n} (|u|+|\nabla u|)\leq C$ and $\nabla u(0)\neq 0$ (in particular, $u$ is nonconstant).

Call now $\Phi_j:\mathbb{C}^n\to\mathbb{C}^n$ the map given by $\Phi_j(z)=C_j^{-1}z + x_j$, so that $\hat{u}_j=u_j\circ \Phi_j$ on $B_{C_j}(0)$.
The K\"ahler form $\omega$ on the chart near $x$ satisfies that
\begin{equation}\label{resc}
C_j^2\Phi_j^*\omega\to \beta,
\end{equation}
smoothly on compact sets of $\mathbb{C}^n$. In particular, $\Phi_j^*\omega\to 0$ smoothly. Similarly, $\Phi_j^*\omega_h\to 0$. For ease of notation, define $\beta_j=C_j^2\Phi_j^*\omega$. Then,
$$\Phi_j^*((\Delta_g u_j)\omega)=\Phi_j^*(C_j^{-2}(\Delta_g u_j) \cdot C_j^2\omega)=(\Delta_{\beta_j} \hat{u}_j) \beta_j
\to (\Delta u)\beta,$$
weakly in $L^p_{\mathrm{loc}}(\mathbb{C}^n)$ of their coefficients. In particular,
$$\Phi_j^*\ti{\omega}_j=\Phi_j^*\left(\omega_h+\frac{1}{n-1}((\Delta_g u_j)\omega-\ddbar u_j)\right)\to P(u),$$
weakly as currents, and since $\Phi_j^*\ti{\omega}_j>0$, it follows that $P(u)\geq 0$ as currents.
The functions $F_j\circ \Phi_j$ are uniformly bounded, and therefore
\begin{equation}\label{vol}
\Phi_j^*\ti{\omega}_j^n=e^{F_j\circ \Phi_j}\Phi_j^*\omega^n\to 0,
\end{equation}
uniformly on compact sets of $\mathbb{C}^n$. From \eqref{c2e} and \eqref{resc} we see that given any $K\subset \mathbb{C}^n$ compact, we have
$$\sup_{K}|\Delta_{\beta_j}\hat{u}_j-\Delta_{\beta}\hat{u}_j|\to 0,$$
as $j\to\infty$. We also have that
\begin{equation} \label{supKbd}
\sup_K |P(\hat{u}_j)-\Phi_j^*\ti{\omega}_j|_{\beta}\leq \sup_K\bigg( \frac{|(\Delta_{\beta_j}\hat{u}_j)\beta_j-(\Delta_\beta \hat{u}_j)\beta|_\beta}{n-1}
+|\Phi_j^*\omega_h|_{\beta}\bigg)
\end{equation}
 converges to zero as $j\to\infty$.
Since $P(\hat{u}_j)$ and $\Phi_j^*\ti{\omega}_j$ are locally uniformly bounded, this together with \eqref{vol} implies that $P(\hat{u}_j)^n$ converges to zero uniformly on compact sets of $\mathbb{C}^n$.

We now use this to conclude that $u$ is maximal. To see this, let $\Omega$ be a bounded domain in $\mathbb{C}^n$ and $v$ a continuous $(n-1)$-PSH function on a larger bounded domain $\Omega'\Supset\Omega$, with $v\leq u$ on $\de\Omega$. Recall that $\hat{u}_j$ converge to $u$ in $C^{1,\alpha}_{\mathrm{loc}}(\mathbb{C}^n)$. Let $\chi_\delta$ be a family of mollifiers, and let $v_\delta=v*\chi_\delta$. Then for $\delta$ small, $v_\delta$ is smooth and $(n-1)$-PSH on $\Omega$, and $v_\delta$ converges to $v$ uniformly on $\ov{\Omega}$ as $\delta\to 0$. Given $\ve>0$ choose $\delta>0$ such that
$v_\delta \leq u+\ve$ on $\de\Omega$, and $\delta\to 0$ as $\ve \to 0$.
Choose $j$ sufficiently large so that $v_\delta\leq \hat{u}_j+2\ve$ on $\de\Omega$.
Let
$$\ti{v}_\ve=v_\delta+\ve|z|^2-\ve C,$$
where $C$ is large enough so that on $\de\Omega$ we have $\ti{v}_\ve\leq v_\delta-2\ve\leq \hat{u}_j$ for all $j$ large.
On $\Omega$ we have
$P(\ti{v}_\ve)=P(v_\delta)+\ve\beta>0,$ and $P(\ti{v}_\ve)^n\geq \ve^n \beta^n$.  From (\ref{supKbd}), we have on $\Omega'$ that $P(\hat{u}_j + \delta_j |z|^2) \ge 0$ where $\delta_j \rightarrow 0$ as $j \rightarrow \infty$.  Moreover, on $\Omega$,  $(P(\hat{u}_j+ \delta_j |z|^2))^n\leq \delta'_j \beta^n$ for $\delta'_j \rightarrow 0$. Therefore we have that $P(\ti{v}_\ve)^n \geq (P(\hat{u}_j+ \delta_j |z|^2))^n$ holds on $\Omega$ for all $j$ large. Since $\ti{v}_\ve$ and $\hat{u}_j+\delta_j |z|^2$ are all smooth, Lemma \ref{comparison} then gives $\ti{v}_\ve\leq\hat{u}_j+\delta_j |z|^2$ on $\Omega$. Letting $j\to\infty$ and then $\ve,\delta\to 0$ shows that $v\leq u$ on $\Omega$ as needed.

Therefore $u$ is maximal, and it satisfies all the hypotheses of the Liouville Theorem \ref{liouv}. We conclude that $u$ is constant, which is a contradiction to $\nabla u(0)\neq 0$.
\end{proof}

Before we prove Theorem \ref{liouv}, we will need two more lemmas.

\begin{lemma}\label{converg}
Let $u_j$ be continuous $(n-1)$-PSH functions on $\mathbb{C}^n$ which are all maximal. Assume that $u_j\to u$ locally uniformly in $\mathbb{C}^n$.
Then $u$ is also continuous, $(n-1)$-PSH and maximal.
\end{lemma}
\begin{proof}
Clearly $u$ is continuous and $(n-1)$-PSH. Let $\Omega$ be a bounded domain in $\mathbb{C}^n$, and $v$ be a continuous $(n-1)$-PSH function on a bounded domain
$\Omega'\Supset\Omega$ with $v\leq u$ on $\de\Omega$, and fix $\ve>0$. Since $u_j\to u$ locally uniformly, for all $j$ large we have $v-\ve\leq u_j$ on $\de\Omega$.
By maximality of $u_j$, we get $v-\ve\leq u_j$ on $\Omega$. Letting $j\to\infty$ and then $\ve\to 0$ shows that $v\leq u$ on $\Omega$ as needed.
\end{proof}

\begin{lemma}\label{maximal}
Let $u$ be a continuous $(n-1)$-PSH function on $\mathbb{C}^n, n\geq 3$. Assume that $u$ is maximal and independent of $z^n$. Then
$u(z^1,\dots,z^{n-1})=u(z^1,\dots,z^{n-1},0)$ is $(n-2)$-PSH and maximal in $\mathbb{C}^{n-1}$.
\end{lemma}
\begin{proof}
Since $u$ does not depend on $z^n$, it is immediate that $u$ is $(n-2)$-PSH as a function on $\mathbb{C}^{n-1}$.

Let $\Omega$ then be a bounded domain in $\mathbb{C}^{n-1}$, contained in the ball $B_R(0) \subset \mathbb{C}^{n-1}$ for some $R>0$.  Suppose that $v$ is a continuous $(n-2)$-PSH function on a bounded domain $\Omega'\Supset\Omega$ with $v \le u$ on $\partial \Omega$.  We need to show that $v\le u$ on $\Omega$.

Fix $\ve>0$.  Let $\ti{\Omega} = \Omega \times \{ |z^n| < C_{\ve} \} \subset \mathbb{C}^n$ for $C_{\ve}$ a constant to be determined.  Note that $\ti{\Omega}$ is a bounded domain in $\mathbb{C}^n$ with boundary $$ \partial \ti{\Omega} = (\partial \Omega \times \{ |z^n | \le C_{\ve} \} ) \cup ( \ov{\Omega} \times \{ |z^n| = C_{\ve} \}).$$
Define
$$v_{\ve}(z^1, \ldots, z^n) = v(z^1, \ldots, z^{n-1}) + \ve(|z^1|^2+ \cdots + |z^{n-1}|^2 - |z^n|^2) - R^2 \ve,$$
which is defined on $\Omega'\times \mathbb{C}\Supset \ti{\Omega}.$
Then $P(v_{\ve}) \ge 0$.  Now on $\partial \Omega \times \{ |z^n| \le C_{\ve}\}$ we have
$$v_{\ve}(z^1, \ldots, z^n) \le v(z^1, \ldots, z^{n-1}) \le u(z^1, \ldots, z^{n-1}).$$
And on $\ov{\Omega} \times \{ |z^n|= C_{\ve} \}$ we have
\[
\begin{split}
v_{\ve} (z^1, \ldots, z^n) = {} &v(z^1, \ldots, z^{n-1}) + \ve (|z^1|^2+ \cdots + |z^{n-1}|^2 - C_{\ve}^2) - R^2\ve \\
\le {} & \sup_{\Omega} v + \ve R^2 - \ve C_{\ve}^2 - R^2 \ve.
\end{split}
\]
Now choose $C_{\ve}$ large enough, depending on  $\ve$, $\sup_\Omega v$ and $\sup_{\Omega}|u|$ to get
$$v_{\ve} \le u \quad \textrm{on } \ov{\Omega} \times \{ |z^n| =C_{\ve}\}.$$
Combining the above, we get $v_{\ve} \le u$ on $\partial \ti{\Omega}$ and so by the maximality of $u$ on $\mathbb{C}^n$ we have $v_{\ve} \le u$ on $\ti{\Omega}$.  Hence on $\Omega$ we have
$$v(z^1, \ldots, z^{n-1}) + \ve(|z^1|^2+ \cdots + |z^{n-1}|^2) - R^2 \ve \le u(z^1, \ldots, z^{n-1}).$$
Letting $\ve \rightarrow 0$ gives $v \le u$ on $\Omega$ and we are done.
\end{proof}

Note that Lemma \ref{maximal} is \emph{false} for plurisubharmonic functions, and a crucial point in our proof is that the function on $\mathbb{C}^n$,
$$z \mapsto |z^1|^2+ \cdots + |z^{n-1}|^2-|z^n|^2,$$
which is not plurisubharmonic, is $(n-1)$-PSH for $n \ge 3$.

\begin{proof}[Proof of Theorem \ref{liouv}]
The rest of the proof follows closely the argument of Dinew-Ko\l odziej \cite{DK}. We will also use their notation. For $u$ a function in $\mathbb{C}^n$,
$z\in\mathbb{C}^n$ and $r>0$, we let
$$[u]_r(z)=\frac{1}{|B_r|}\int_{B_r(z)}u \beta^n,$$
be the average of $u$ on the ball $B_r(z)$. Similarly, if $\alpha=\sum_{I,J}\alpha_{I\ov{J}} dz^I\wedge d\ov{z}^J$ is a $(p,q)$ form (here $I,J$ are multi-indices with $|I|=p, |J|=q$), we define
$$[\alpha]_r(z)=\sum_{I,J}[\alpha_{I,J}]_r(z) dz^I\wedge d\ov{z}^J,$$
which is another $(p,q)$ form. We have that
$$\left(\frac{\de}{\de z}[u]_r\right)(z)= \left[\frac{\de u}{\de z}\right]_r(z), \ \de [\alpha]_r = [\de\alpha]_r \ \textrm{and } [P(u)]_r=P([u]_r),$$
for any function $u$ and $(p,q)$ form $\alpha$.

We prove the Liouville theorem by induction on $n\geq 2$. When $n=2$, a function with $P(u)\geq 0$ is plurisubharmonic, and it is well-known that every bounded plurisubharmonic function on $\mathbb{C}^n$ is constant. So we can assume that the result holds in dimension $n-1$ and prove it in dimension $n$.
For a contradiction, we assume that our maximal function $u$ is nonconstant, and we normalize it so that $\inf_{\mathbb{C}^n}u=0$ and $\sup_{\mathbb{C}^n}u=1$.
We noted earlier that every $(n-1)$-PSH function is in particular subharmonic, and therefore $u$ and $u^2$ are subharmonic and bounded on $\mathbb{C}^n$.
As discussed in \cite{DK}, the Harnack inequality (or Cartan's lemma) implies that for every $z\in\mathbb{C}^n$ we have
\begin{equation}\label{lim}
\lim_{r\to\infty} [u]_r(z)=\lim_{r\to\infty} [u^2]_r(z)=1.
\end{equation}
Consider the following property $(*)$: there exist a radius $\rho>0$, a sequence of maps $G_k:\mathbb{C}^n\to\mathbb{C}^n$ of the form
$G_k(z)=H_kz+\lambda_k$ with $H_k\in U(n)$ and $\lambda_k\in\mathbb{C}^n$, and a sequence of radii $r_k\to\infty$ such that
\begin{equation}\label{cond1}
[u^2\circ G_k]_{r_k}(0)+[u\circ G_k]_{\rho}(0)-2u\circ G_k(0)\geq\frac{4}{3},
\end{equation}
and
\begin{equation}\label{cond2}
\lim_{k\to\infty} \int_{B_{r_k}(0)}\left|\frac{\de(u\circ G_k)}{\de z^n}\right|^2\beta^n=0.
\end{equation}
We will show that both $(*)$ and its negation lead to a contradiction. Assume first that $(*)$ holds. Let $u_k=u\circ G_k$ which are still defined on the whole of $\mathbb{C}^n$ and are Lipschitz and maximal. Then we have
$$\sup_{\mathbb{C}^n}|u_k|=\sup_{\mathbb{C}^n}|u|\leq C,\quad \sup_{\mathbb{C}^n}|\nabla u_k|=\sup_{\mathbb{C}^n}|\nabla u|\leq C,$$
and by Ascoli-Arzel\`a and a diagonalization argument, a subsequence of $u_k$ converges locally uniformly to a continuous function $v$ on $\mathbb{C}^n$, which is clearly $(n-1)$-PSH and Lipschitz with
$\sup_{\mathbb{C}^n}|\nabla v|\leq C$. Lemma \ref{converg} shows that $v$ is maximal.

Exactly the same argument as in \cite{DK}, which uses \eqref{cond2}, shows that $v$ is independent of $z^n$, and Lemma \ref{maximal} shows that $v$ is maximal in $\mathbb{C}^{n-1}$. By induction hypothesis, $v$ is constant.
But going back to \eqref{cond1} we can use that $[u^2\circ G_k]_{r_k}(0)\leq 1$ to see that
$$[v]_\rho(0)-2v(0)\geq \frac{1}{3}>0,$$
which is absurd since $v$ is a constant between 0 and 1.

We can then assume that the negation of $(*)$ is true, i.e. there is a constant $c>0$ such that for all $\rho>0$, all $G=H+\lambda,$ ($H\in U(n), \lambda\in\mathbb{C}^n$), and all $r$ large, either
$$[u^2\circ G]_{r}(0)+[u\circ G]_{\rho}(0)-2u\circ G(0)<\frac{4}{3},$$
or
$$\int_{B_{r}(0)}\left|\frac{\de(u\circ G)}{\de z^n}\right|^2\beta^n\geq c>0.$$
In particular, given any point $z\in\mathbb{C}^n$ and unit vector $w\in\mathbb{C}^n$, we can choose $G$ such that $G(0)=z$ and $G$ maps $\frac{\de}{\de z^n}$ to $w$. Therefore
there exist $c>0$ and $R>0$ such that if $r>R$, $z\in\mathbb{C}^n$ and $\rho>0$ satisfy
$$[u^2]_r(z)+[u]_\rho(z)-2u(z)\geq\frac{4}{3},$$
then we have that
\begin{equation}\label{key}
\left[ \left|\frac{\de u}{\de w}\right|^2\right]_r(z) \geq c\cdot c_nr^{-2n}>0,
\end{equation}
for all unit vectors $w$.

Shift the origin so that $u(0)<\frac{1}{12}$, and then pick $\rho$ and $r>R$ large so that $[u]_\rho(0)>\frac{3}{4}$ and $[u^2]_r(0)>\frac{3}{4}$, which is possible thanks
to \eqref{lim}.
Then
$$[u^2]_r(0)+[u]_\rho(0)-2u(0)>\frac{3}{4}+\frac{3}{4}-\frac{1}{6}=\frac{4}{3}.$$
Define an open set
$$U=\bigg\{2u<[u^2]_r+[u]_\rho -\frac{4}{3}\bigg\},$$
which is nonempty since $0\in U$, and for every $z\in U$ we have
that \eqref{key} holds for all unit vectors $w$.
Compute
$$P(u^2)=2uP(u)+2|\nabla u|^2\beta-2\mn\de u\wedge\db u\geq 2|\nabla u|^2\beta-2\mn\de u\wedge\db u,$$
as currents.
We claim that there is a constant $c'>0$ such that if $z\in U$ then
$$\left[|\nabla u|^2\beta-\mn\de u\wedge\db u\right]_r(z)\geq c'\beta.$$
To see this, write $|\nabla u|^2\beta-\mn\de u\wedge\db u=\mn\alpha_{i\ov{j}} dz^i\wedge d\ov{z}^j$, so that
$$\alpha_{i\ov{i}}=\sum_{j\neq i}\left|\frac{\de u}{\de z^j}\right|^2,$$
at a.e. point (where $u$ is differentiable).
We need a lower bound for $[\alpha_{i\ov{j}}]_r(z) \mn dz^i\wedge d\ov{z}^j$. We can choose the coordinates $(z^1,\dots,z^n)$ so that
the Hermitian matrix $[\alpha_{i\ov{j}}]_r(z)$ is diagonal, with eigenvalues
$$[\alpha_{i\ov{i}}]_r(z)=\frac{1}{|B_r|}\int_{B_r(z)}\alpha_{i\ov{i}}\, \beta^n=\frac{1}{|B_r|}\sum_{j\neq i}
\int_{B_r(z)}\left|\frac{\de u}{\de z^j}\right|^2 \beta^n=\sum_{j\neq i} \left[ \left|\frac{\de u}{\de z^j}\right|^2\right]_r(z),$$
which are all bigger than $c'=c'(r)>0$ thanks to \eqref{key}. This proves the claim, and so we get
$$P([u^2]_r)\geq c'\beta.$$
Define $\gamma(z) = - \frac{c'}{2} |z|^2$ which has the property that $P(\gamma) = -\frac{c'}{2} \beta$.
Then the function $[u^2]_r+\gamma$ is continuous and $(n-1)$-PSH on $U$. Consider the open set $U_{\gamma} \subset U$ given by
$$U_{\gamma}=\bigg\{2u<[u^2]_r+[u]_\rho+\gamma -\frac{4}{3}\bigg\}.$$
Then $0\in U_{\gamma}$, $U_{\gamma}$ is bounded since $u$ is bounded and $\gamma \to -\infty$ as $|z|\to\infty$, and in fact $U_{\gamma}\Subset U$.
The functions $[u^2]_r+[u]_\rho+\gamma -\frac{4}{3}$ and $2u$ are continuous and $(n-1)$-PSH on $U$, and they are equal on $\de U_{\gamma}$.
By maximality of $2u$ we conclude that $2u\geq [u^2]_r+[u]_\rho+\gamma -\frac{4}{3}$ inside $U_{\gamma}$, which is absurd. This contradiction ends the proof of the theorem.
\end{proof}

\section{Proof of the Main Theorem} \label{sectionproof}

To complete the proof of Theorem \ref{maintheorem}, we need to establish higher order estimates and set up a suitable continuity method argument.  Using the estimates obtained in the previous sections, it is now straightforward to establish $C^{\infty}$ a priori estimates for $u$ solving (\ref{MA1}).  In fact, by the correspondence between equations (\ref{MA1}) and (\ref{FWW}), this result is already contained in the work of Fu-Wang-Wu \cite{FWW}.  Nevertheless, we include a short proof for the sake of completeness.

\begin{theorem} \label{theoremhigher}
In the setting of Theorem \ref{maintheorem}, let $u$ solve (\ref{MA1}) subject to (\ref{condthm}).  Then for each $k=0,1,2, \ldots$ there exists a positive constant $C_k$ depending only on $k$ and the fixed data $(M, g, h)$ and bounds for $F$ such that
$$\| u \|_{C^k(M,g)} \le C_k,$$
and $$\tilde{g} \ge \frac{1}{C_0} g,$$
for $\tilde{g}$ given by (\ref{tildeg1}).
\end{theorem}
\begin{proof}
Combining the results of Theorems \ref{theoremC2} and \ref{grad}, we have the following estimates:
$$\sup_M |u| + \sup_M | \partial u|_g + \sup_M | \partial \ov{\partial} u |_g \le C,$$
and from the equation (\ref{MA1}), the uniform upper bound on $\tr{g}{\tilde{g}}$ gives $$C^{-1} g \le \tilde{g} \le C g.$$

By the standard linear elliptic theory, it suffices to obtain a $C^{2+\alpha}(M,g)$ bound for $u$ for some $\alpha>0$.
 We  now apply the usual Evans-Krylov method, adapted to the complex setting (see \cite{Si, Tr}).  Here we will follow the proof given in \cite{TW1} for the complex Monge-Amp\`ere equation on Hermitian manifolds.
 We will describe how that proof goes through here with some  minor modifications.
Recall that our equation is
$$\log \det \tilde{g} = \tilde{F}.$$
As in \cite{TW1}, we will work in a small open subset of $\mathbb{C}^n$, containing a ball $B_{2R}$ of radius $2R$.
Let $\gamma=(\gamma^i)$ be an arbitrary vector in $\mathbb{C}^n$.  From (\ref{F112}), but replacing the derivatives in the $\frac{\partial}{\partial z^1}$ direction with derivatives in the $\gamma^i \partial/\partial z^i$ direction, we obtain
\begin{equation} \label{ek1}
\Theta^{i\ov{j}} \nabla_{\ov{j}} \nabla_i u_{\gamma \ov{\gamma}} \ge G,
\end{equation}
for $G$ a uniformly bounded function.  Recall that $\Theta^{i\ov{j}}$ is defined by
$$\Theta^{i\ov{j}} = \frac{1}{n-1} \left( (\tr{\tilde{g}}{g}) g^{i\ov{j}} - \tilde{g}^{i\ov{j}} \right).$$
Of course the metrics $\tilde{g}_{i\ov{j}}$ and $\Theta^{i\ov{j}}$ are uniformly bounded and equivalent to a Euclidean metric on $B_{2R}$.
  Observe that we are working here with \emph{covariant derivatives} with respect to $g$, and not partial derivatives as in \cite{TW1}.

 For the next step, we will encounter a minor complication arising from the fact that the operator
  $$u \mapsto (\Delta u) g_{i\ov{j}} - u_{i\ov{j}}$$
  depends on the metric $g$, which is varying on $B_{2R}$.  For this reason we introduce the fixed metric $\hat{g}_{i\ov{j}}$ on $B_{2R}$ to be the positive definite Hermitian matrix
  $$\hat{g}_{i\ov{j}} = g_{i\ov{j}}(0).$$
 Regarding $\hat{g}$ as a K\"ahler metric on $B_{2R}$, we have, by the Mean Value Inequality, the bounds
 \begin{equation} \label{hatgg}
 | \hat{g}_{i\ov{j}}(x) - g_{i\ov{j}}(x) | \le CR, \quad \textrm{for } x \in B_{2R}.
 \end{equation}
 We then define
 $$\hat{\Theta}^{i\ov{j}} = \frac{1}{n-1} \left( (\tr{\tilde{g}}{\hat{g}}) \hat{g}^{i\ov{j}} - \tilde{g}^{i\ov{j}} \right).$$
 Now the concavity of the operator $\Phi(A) = \log \det A$ (for $A$ a positive definite Hermitian matrix) implies that, for $x,y \in B_{2R}$,
 $$\Phi(\tilde{g}(y)) + \sum_{i,j} \frac{\partial \Phi}{\partial a_{i\ov{j}}} (\tilde{g}(y)) \left( \tilde{g}_{i\ov{j}}(x) - \tilde{g}_{i\ov{j}}(y) \right) \ge \Phi(\tilde{g}(x)).$$
Hence
 $$\sum_{i,j} \tilde{g}^{i\ov{j}}(y) \left( \tilde{g}_{i\ov{j}}(y) - \tilde{g}_{i\ov{j}}(x) \right) \le \ti{F}(y)-\ti{F}(x) \le CR.$$
But this implies that
$$  \frac{1}{n-1} \sum_{i,j} \tilde{g}^{i\ov{j}}(y) \left( ( (\Delta u) g_{i\ov{j}} - u_{i\ov{j}})(y) - ( (\Delta u) g_{i\ov{j}} - u_{i\ov{j}})(x) \right)  \le C'R.$$
From (\ref{hatgg}) and the fact that the $u_{i\ov{j}}$ and $\tilde{g}^{i\ov{j}}$ are uniformly bounded, we obtain
$$ \frac{1}{n-1} \sum_{i,j} \tilde{g}^{i\ov{j}}(y) \left( ( (\hat{\Delta} u) \hat{g}_{i\ov{j}} - u_{i\ov{j}})(y) - ( (\hat{\Delta} u) \hat{g}_{i\ov{j}} - u_{i\ov{j}})(x) \right)  \le C'R,$$
where we are writing $\hat{\Delta}=\sum_{i,j} \hat{g}^{i\ov{j}} \partial_i \partial_{\ov{j}}$.  It follows that
\begin{equation} \label{ek2}
\sum_{i,j} \hat{\Theta}^{i\ov{j}} (y) \left( u_{i\ov{j}}(y) - u_{i\ov{j}}(x) \right) \le C'R.
\end{equation}
Given (\ref{ek1}) and (\ref{ek2}), we can apply the arguments of \cite{TW1} in exactly the same way.  The only difference is that in (\ref{ek1}), we are using covariant derivatives instead of partial derivatives.  However, the difference between the operators $\Theta^{i\ov{j}} \partial_{\ov{j}} \partial_i$ and $\Theta^{i\ov{j}} \nabla_{\ov{j}} \nabla_i$ are some first and zero order terms with bounded coefficients. We can still apply Theorem 9.22 in \cite{GT} to obtain the result of Lemma 4.1 in \cite{TW1}.  The result follows.
\end{proof}

It remains to set up a continuity method and establish ``openness'', and prove the uniqueness of the solution.  Again these already follow from the results of Fu-Wang-Wu \cite{FWW} in the case of the equation (\ref{FWW}).   We include here our own proofs which use directly the equation (\ref{MA1}).

\begin{proof}[Proof of Theorem \ref{maintheorem}]
Given a smooth function $F$ we will find a constant $b$ and a function $u$  such that
\begin{equation} \label{MA2}
\left(\omega_h+ \frac{1}{n-1} \left( (\Delta u) \omega - \ddbar u\right)\right)^n = e^{F+b} \omega_h^n,
\end{equation}
and
\begin{equation} \label{MAC2}
\omega_h+ \frac{1}{n-1} \left( (\Delta u) \omega - \ddbar u\right) >0, \quad \sup_M u=0,
\end{equation}
for a Hermitian metric $\omega_h$ and a K\"ahler metric $\omega$.  Note that this equation differs slightly from (\ref{MA1}) since we have replaced $\omega^n$ on the right hand side by $\omega_h^n$.  However, we are free to make this change as it corresponds to adding a smooth bounded function to $F$.

 By the  higher order estimates, it suffices to find $u \in C^{2+\alpha}$ for some $\alpha$ with $0<\alpha<1$, which we now fix.
As in \cite{TW1} we consider a family of equations for $u_t$, $b_t$,
\begin{equation} \label{family}
\left(\omega_h+ \frac{1}{n-1} \left( (\Delta u_t) \omega - \ddbar u_t \right)\right)^n = e^{tF+b_t} \omega_h^n,
\end{equation}
with
\begin{equation} \label{MAC2t}
\omega_h+ \frac{1}{n-1} \left( (\Delta u_t) \omega - \ddbar u_t\right) >0, \quad \sup_M u_t=0.
\end{equation}
Consider the set
$$ T =\left\{ t' \in [0,1] \ \bigg| \  \begin{array}{l} \textrm{there exists $u_t \in C^{2+\alpha}(M)$ and $b_t$} \\ \textrm{solving (\ref{family}), (\ref{MAC2t}) for $t\in [0,t']$ } \end{array} \right\}.$$
Note that $0\in T$.  We wish to show that $T$ is open. Assume that $\hat{t} \in T$.
We write
$$\hat{\omega} = \omega_h + \frac{1}{n-1}( (\Delta u_{\hat{t}}) \omega - \ddbar u_{\hat{t}}).$$
It suffices to show that, for some small $\ve>0$, there exists $v_t \in C^{2 + \alpha}(M)$  for $t \in [\hat{t}, \hat{t}+\ve)$ with $v_{\hat{t}}=0$ and
$$\left(\hat{\omega} + \frac{1}{n-1} ((\Delta v_t) \omega - \ddbar v_t)\right)^n = e^{(t-\hat{t}\, )F + b_t -b_{\hat{t}}} \hat{\omega}^n,$$
and
$$\hat{\omega} + \frac{1}{n-1} \left( (\Delta v_t) \omega - \ddbar v_t\right) >0,$$
for $b_t$ a function of $t$. Indeed, if we can find such a $v_t$ then $u_t = u_{\hat{t}} + v_t$ solves (\ref{family}) (and by adding  a time-dependent constant, we can also arrange that $\sup_M u_t=0$). Now define
\begin{equation*}
\hat{\Theta}^{i\ov{j}} = \frac{1}{n-1} ((\tr{\hat{g}}{g}) g^{i\ov{j}} - \hat{g}^{i\ov{j}}).
\end{equation*}
By Gauduchon's theorem \cite{Ga}, there exists a smooth function $\sigma$ such that
\begin{equation} \label{defnsigma}
e^{\sigma} \hat{\Theta}^{i\ov{j}} \hat{\omega}^n = g_G^{i\ov{j}} \omega_G^n,
\end{equation}
where $\omega_G= \sqrt{-1} (g_G)_{i\ov{j}} dz^i \wedge d\ov{z}^j$ is a Gauduchon metric.  Indeed, writing $\hat{\Theta}_{i\ov{j}}$ for the Hermitian metric which is the inverse of $\hat{\Theta}^{i\ov{j}}$, there exists a smooth function $\sigma'$ with $(g_{G})_{i\ov{j}} = e^{\sigma'} \hat{\Theta}_{i\ov{j}}$ a Gauduchon metric, thanks to \cite{Ga}.  Then set $\sigma = - \sigma' + \log (\omega_G^n/\hat{\omega}^n)$.
 By adding a constant to $\sigma$, we may and do assume that
$\int_M e^{\sigma} \hat{\omega}^n=1$.

We will show that we can find $v_t$ for $t \in [\hat{t}, \hat{t}+\ve)$ such that
\[
\begin{split}
\lefteqn{
\left(\hat{\omega} + \frac{1}{n-1} ( (\Delta v_t)\omega- \ddbar v_t) \right)^n} \\ = {} & \left( \int_M e^{\sigma} (\hat{\omega} + \frac{1}{n-1}((\Delta v_t) \omega - \ddbar v_t))^n \right) e^{(t-\hat{t}\, )F + c_t} \hat{\omega}^n,
\end{split}
\]
where $c_t$ is the normalization constant given by
$$\int_M e^{(t-\hat{t}\, ) F+c_t} e^{\sigma} \hat{\omega}^n =1.$$
Define $B_1$, $B_2$ by
$$B_1 = \{ v \in C^{2+\alpha}(M) \ | \ \int_M v e^{\sigma} \hat{\omega}^n =0, \ \hat{\omega} + \frac{1}{n-1} ((\Delta v) \omega - \ddbar v) >0 \}$$
and
$$B_2 = \{ w \in C^{\alpha}(M) \ | \ \int_M e^w e^{\sigma} \hat{\omega}^n =1 \}.$$
Define $\Psi: B_1 \rightarrow B_2$ by
\[
\begin{split}
\lefteqn{\Psi(v) = \log \frac{ (\hat{\omega} + \frac{1}{n-1} ((\Delta v)\omega - \ddbar v))^n}{\hat{\omega}^n}  } \qquad \\ & - \log \left( \int_M e^{\sigma} (\hat{\omega} + \frac{1}{n-1}((\Delta v)\omega - \ddbar v))^n \right).
\end{split}
\]
Our goal is to find $v_t$ solving $\Psi(v_t) = (t-\hat{t}\, )F + c_t$ for $t \in [\hat{t}, \hat{t}+\ve)$.  We have $\Psi(0)=0$ and so by the Inverse Function Theorem, it suffices to show the invertibility of
$$(D\Psi)_0:  T_0 B_1\rightarrow T_0 B_2,$$
where
$$T_0 B_1= \{ \zeta\in C^{2+\alpha}(M) \ | \ \int_M \zeta e^{\sigma} \hat{\omega}^n =0 \},$$
and
$$T_0 B_2 = \{ \rho\in C^{\alpha}(M) \ | \ \int_M \rho e^{\sigma} \hat{\omega}^n=0 \},$$ denote the tangent spaces to $B_1, B_2$ at 0.
We have
$$(D\Psi)_0 (\zeta) = \hat{\Theta}^{i\ov{j}} \zeta_{i\ov{j}} - \int_M e^{\sigma} \hat{\Theta}^{i\ov{j}} \zeta_{i\ov{j}} \hat{\omega}^n= \hat{\Theta}^{i\ov{j}} \zeta_{i\ov{j}},$$
since from (\ref{defnsigma}),
$$ \int_M e^{\sigma} \hat{\Theta}^{i\ov{j}} \zeta_{i\ov{j}} \hat{\omega}^n = \int_M \Delta_{\omega_G} \zeta \, \omega_G^n=0.$$
This operator is clearly injective. To show surjectivity, take $\rho \in T_0B_2$.  Then let $\zeta$ solve
$$\Delta_{\omega_G} \zeta = \rho e^{\sigma} \frac{\hat{\omega}^n}{\omega_G^n}, \quad \int_M \zeta e^{\sigma} \hat{\omega}^n=0.$$
We can find this $\zeta$ since the integral of $\rho e^{\sigma} \hat{\omega}^n/\omega_G^n$ with respect to $\omega_G^n$ is zero, and $\omega_G$ is Gauduchon (and furthermore, the
$C^{2+\alpha}$ norm of $\zeta$ is bounded by the $C^{\alpha}$ norm of $\rho$).
But  this means that
$$\frac{\omega_G^n}{\hat{\omega}^n} e^{-\sigma} \Delta_{\omega_G} \zeta = \rho$$
and hence by (\ref{defnsigma}),
$$\hat{\Theta}^{i\ov{j}} \zeta_{i\ov{j}} = \rho,$$
which is exactly what we needed to show.

This establishes the openness of the set $T$.  For the closedness, we need a priori estimates for $u_t$ and $b_t$.  Bounds on $b_t$ follow immediately from the maximum principle as discussed in the beginning of Section \ref{linfty}, and then the estimates for $u_t$ follow from Theorem \ref{theoremhigher}.  Hence $T$ is open and closed and so equal to $[0,1]$.  The solution at $t=1$ gives us a solution to (\ref{MA2}).

Uniqueness of the solution $(u,b)$ to (\ref{MA1}) follows from the arguments of \cite{FWW}.  Indeed, suppose that $(u,b)$ and $(u',b')$ are two solutions of (\ref{MA1}).  Write
$$\tilde{\omega} = \omega_h + \frac{1}{n-1}( (\Delta u) \omega - \ddbar u)$$
and
$$\tilde{\omega}' = \omega_h + \frac{1}{n-1}( (\Delta u') \omega - \ddbar u').$$
Then
$$\frac{(\tilde{\omega} + \frac{1}{n-1} ((\Delta (u'-u)) \omega - \ddbar (u'-u)))^n}{ \tilde{\omega}^n} = e^{b'-b},$$
and by considering the points at which $u'-u$ achieves a maximum and minimum, we obtain $b=b'$.  To see that $u=u'$, observe that $\tilde{\omega}^n = (\tilde{\omega}')^n$ and hence
$$ \left( \Delta (u-u') - \ddbar (u-u') \right) \wedge \sum_{i=0}^{n-1} \tilde{\omega}^i \wedge (\tilde{\omega}')^{n-1-i}=0.$$
Since $\sup_M u=\sup_M u'=0$, it follows that $u-u'=0$ by the strong maximum principle.  This completes the proof of Theorem \ref{maintheorem}.
\end{proof}

\end{document}